\documentclass[a4paper,11pt,reqno]{amsart}
\usepackage[foot]{amsaddr}
\usepackage{amssymb}
\usepackage{epsfig}
\usepackage{amsfonts}
\usepackage{amsmath}
\usepackage{euscript}
\usepackage{amscd}
\usepackage{amsthm}
\usepackage{enumitem}
\DeclareMathAlphabet{\mathpzc}{OT1}{pzc}{m}{it}
\usepackage{enumitem}
\usepackage{color}
\usepackage{mathtools}
\usepackage{xcolor}
\usepackage{soul}
\usepackage{graphicx}
\usepackage{wrapfig}
\usepackage{subcaption}
\usepackage[nopar]{lipsum}
\usepackage[pagebackref, hypertexnames=false, colorlinks, citecolor=red, linkcolor=blue, urlcolor=red]{hyperref}
\usepackage{cleveref}
\usepackage[marginratio=1:1,height=584pt,width=380pt,tmargin=112pt]{geometry}
\usepackage{pgf,tikz,pgfplots}
\usetikzlibrary{shapes,arrows}
\usepackage{adjustbox}
\usepackage{setspace}
\usepackage[numbers,sort&compress]{natbib}
\usepackage{algorithm2e}%for algorithm with \caption{...}\ref{...}
\RestyleAlgo{ruled}% ruled environment around algorithm

\usetikzlibrary{calc}
\usetikzlibrary{patterns}
\pgfplotsset{compat=1.14}
\tikzstyle{none}=[]
\tikzstyle{new style 0}=[draw,circle,fill=white]
\tikzstyle{new edge style 1}=[draw,dashed]
\tikzstyle{new edge style 1}=[draw,dashed]
\tikzstyle{new edge style 0}=[->]

\usepackage{marginnote}

\usepackage[normalem]{ulem}
\newcommand{\updownextend}[1]{
	\addtolength{\textheight}{#1}}
\updownextend{2cm}
\allowdisplaybreaks[4]

\usepackage{mathrsfs}
\DeclareFontFamily{OT1}{pzc}{}
\DeclareFontShape{OT1}{pzc}{m}{it}{<-> s * [1.10] pzcmi7t}{}
\DeclareMathAlphabet{\mathpzc}{OT1}{pzc}{m}{it}

\DeclareSymbolFont{SY}{U}{psy}{m}{n}
\DeclareMathSymbol{\emptyset}{\mathord}{SY}{'306}

\theoremstyle{plain}

\newtheorem{thm}{Theorem}[section]
\newtheorem*{thm*}{Theorem}
\newtheorem{cor}[thm]{Corollary}
\newtheorem{lem}[thm]{Lemma}
\newtheorem{prop}[thm]{Proposition}
\newtheorem{defn}[thm]{Definition}

\newtheorem{ex}[thm]{Example}

\newtheoremstyle{named}{}{}{\itshape}{}{\bfseries}{.}{.5em}{#1 \thmnote{#3}}
\theoremstyle{named}

\numberwithin{equation}{section}
\def\C{{\mathbb C}}

\def\l{\lambda}

\def\ov{\overline}

\def\m{\mathcal}
\def\mf{\mathfrak}
\def\mb{\mathbb}

\def\G{\Gamma}
\def\d{\sum}

\def\beq{\begin{eqnarray}}
	\def\eeq{\end{eqnarray}}
\def\beqa{\begin{eqnarray*}}
	\def\eeqa{\end{eqnarray*}}

\def\ov{\overline}

\newcommand{\be}{\begin{equation}}
	\newcommand{\ee}{\end{equation}}
\newcommand{\bea}{\begin{eqnarray}}
	\newcommand{\eea}{\end{eqnarray}}
\newcommand{\Bea}{\begin{eqnarray*}}
	\newcommand{\Eea}{\end{eqnarray*}}

\newcounter{cnt1}
\newcounter{cnt2}
\newcounter{cnt3}
\newcommand{\blr}{\begin{list}{$($\roman{cnt1}$)$}
		{\usecounter{cnt1} \setlength{\topsep}{0pt}
			\setlength{\itemsep}{0pt}}}
	\newcommand{\bla}{\begin{list}{$($\alph{cnt2}$)$}
			{\usecounter{cnt2} \setlength{\topsep}{0pt}
				\setlength{\itemsep}{0pt}}}
		\newcommand{\bln}{\begin{list}{$($\arabic{cnt3}$)$}
				{\usecounter{cnt3} \setlength{\topsep}{0pt}
					\setlength{\itemsep}{0pt}}}
			\newcommand{\el}{\end{list}}

		\DeclareMathOperator  {\tr} {tr}

		\newcommand{\com}[1]{\mathcal{C}_{#1}}%coomutative graph of #1

\begin{document}
			
			\title{On the spectra of Commuting Graphs}
			\author[Ghosh, Parui]{Gargi Ghosh and Samiron Parui}
			\address[Ghosh]{Silesian University in Opava, 746 01, Czech Republic}
			\address[Parui]{Indian Institute of Science Education and Research Kolkata, Mohanpur, 741246, India}
			\email[Ghosh]{gargighosh1811@gmail.com}
			\email[Parui]{samironparui@gmail.com}
			\maketitle
			\begin{abstract}
				We describe the complete spectra of Laplacian, signless Laplacian, and adjacency matrices associated with the commuting graphs of a finite group using group theoretic information. %The order of a group is a Laplacian eigenvalue of the commuting graph of the group.
				We provide a method to find the center of a group by only using the Laplacian of the commuting graph of the group.
				The graph invariants (such as diameter, clique number, and mean distance) of the commuting graph associated with a finite group are determined.  We produce a number of examples  to illustrate the applications of our results.
			\end{abstract}
			\section{Introduction}The interplay between groups and associated graphs (such as  power graph, commuting graph, and enhanced power graph) has been studied extensively. %To explore a group using graphs, one can define graphs whose vertex set is the group. 
			Naturally, the set of all edges of any such graph encodes various properties of the group from the binary interrelation of the group elements \cite{cameron2021graphs}.  
			The commuting graph associated to a group $G$ is the graph on the vertex set $G$, encoding the information about the commutative pairs of elements of $G$. More explicitly, a \emph{commuting graph associated to the group $G$}, denoted by $\com{G}$, is an undirected graph with vertex set $G$ and%the pair
			$\{g_1,g_2\}\subseteq G$ forms an edge in $\com{G}$ if and only if $g_1g_2=g_2g_1$.  The notion of commuting graphs is introduced in \cite[p. 565]{MR74414}. In this article, we define the commuting graph associated with a finite group following \cite{cameron2020between,cameron2021graphs}. 
			
			Our primary aim is to decode various group-theoretic information from the spectra of the adjacency, the Laplacian, and the signless Laplacian of a commuting graph. We use the term \emph{``spectra of a commuting graph"} to denote these three spectra together. The Laplacian spectrum of a graph contains inherent information about various graph invariants such as connectivity, diameter, isoperimetric number, maximum cut, independence number, genus,  mean distance, and bandwidth-type parameters, see \cite{MR2571608,MR1170831,MR1105467,MR318007,MR1275613}. Moreover, the investigation of the Laplacian spectrum enriches the study of natural phenomena involving diffusion, such as synchronization, random walk, consensus \cite{MR4079051,MR1910670,MR1877614,MR3801977,MR4033495}. The spectra of a commuting graph are well-studied in the literature, see \cite{cameron2021graphs,MR4461661,MR4443447,MR4338337,MR4296337} and the references therein. To the best of our knowledge, there is not much information available on the complete spectra of the commuting graph of a finite group, albeit on particular families of groups \cite{MR4553232,MR4338337,MR4041692,MR3769593}. However, more specific spectral information can be derived for various families of groups, cf. \Cref{all-pendant}, \Cref{non_nbd} and \Cref{non_nbd_q}. 
			
			We end this section with a highlight of the key results and ideas.
			\begin{itemize}[leftmargin=*]
				%\item Firstly, \Cref{spectra} is dedicated to a detailed description of the spectra of the commuting graph $\m C_G$ associated to a finite group $G$. 
				\item %In \Cref{complete-l}, we provide the complete spectrum of the Laplacian of $\com{G}$ for a finite group $G$. %We show that $|G|$ is a Laplacian eigenvalue of $\com{G}$. If there exists $u, v(\ne u)\in G$ with $C(u)=C(v)$ then $|C(u)|$ is a Laplacian eigenvalue of $\com{G}$. {\color{blue}Using this fact in \Cref{equalcentralizer}, we provide a spectral condition to identify pairs of $u, v(\ne u)\in G$ such that $C(u)=C(v)$.}% can not decide if this line is important enough to be highlighted
				%\item  %We prove that the order of the group, $|G|$ is an eigenvalue of the Laplacian of $\m C_G$. If there exists distinct $u,v\in G\setminus Z(G)$, such that $u$ and $v$ have the same centralizer, then the cardinality of the centralizer is a Laplacian eigenvalue of the Commuting graph of $G$ (see \Cref{complete-l}). Our study reveals that if $n$ is not a Laplacian eigenvalue of a graph $\G$ on $n$ vertices, then $\G$ can not be the commuting graph of any group (see \Cref{nongl}). In \Cref{equalcentralizer}, we provide a spectral condition to characterize the situation when a given pair of elements in a group can not share the same centralizer. 
				Let $G$ be a finite group. A complete description of the spectra of the Laplacian, the signless Laplacian and the adjacency of $\com{G}$ are provided in \Cref{complete-l}, \Cref{complete-q} and \Cref{complete-a}, respectively. %If $|Z(G)|>1$, then $ |G|-2$ is an eigenvalue of the signless Laplacian of the commuting graph.
				Some eigenvalues of the spectra of $\m C_G$ can be deduced directly whereas for the rest, we define an equivalence relation on $G$ in \Cref{equidef} and produce a quotient graph $ \Gamma_{G/\mathfrak{G}}$ (which is `smaller' than $\com{G}$).  Consequently, it enables to conclude that the eigenvalues of the associated matrices of the quotient graph are the eigenvalues for the corresponding matrices of $\m C_G.$  %In \Cref{l-contraction}, \Cref{q-contraction}, and \Cref{a-contraction}, we show that {\color{blue}the quotient graph $ \Gamma_{G/\mathfrak{G}}$ gives quotient matrices of Laplacian, signless Laplacian, and adjacency of $\com{G}$ such that the eigenvalues of these quotient matrices are eigenvalues of the original matrices}.%For all three matrices associated with the commuting graphs, we observe that some eigenvalues are related to  the centralizers of some elements of the group. Other than these eigenvalues, the rest of the eigenvalues can be calculated from smaller matrices. These smaller matrices are associated with a graph obtained by identifying each equivalence class of an equivalence relation (see \Cref{equidef}) on the group. We describe how the eigenvalues of these smaller matrices give the eigenvalues of Adjacency, Laplacian, and signless Laplacian in \Cref{l-contraction}, \Cref{q-contraction}, and \Cref{a-contraction}.
				\item  On the other hand, \Cref{center_LG} yields a spectral method to compute the center of a finite group $G$  which helps us to figure out the center of $G$ from one Laplacian eigenvalue of $\com{G}$ and its eigenspace.
				%to figure out the center of $G$from the eigenvalue $|G|$ of the Laplacian of $\com{G}$, and its eigenspace
				
				%We devoted the \Cref{Groupinfo} to explore the group information encoded in the spectra of the commuting graph. \Cref{center_LG} is a key result of this section. The result provides a spectral method to compute the center of a group. This leads to the Algorithm(\ref{alg:Z(G)}), which can compute the center of a group from the  eigenvalues and eigenvectors of Laplacian of the commuting graph.
				\item To emphasize our results, we provide spectra of the commuting graphs associated with a group $G\in \mathcal G,$ where the family
				$\mathcal G$ of finite groups includes the finite reflection groups (for example, the symmetric groups, the dihedral groups), the centralizer abelian groups and the quaternion group.
				\item In \Cref{app}, we determine some relations between some graph invariants such as diameter, clique number, mean distance of $\m C_G$ and the group properties such as order of the group, centre of the group etc.
			\end{itemize}
			
			%In this paper, we study the spectrum of the Laplacian matrix of the commuting graph $\mathcal C_G$ associated with a finite group $G$. %Laplacian \cite{MR2571608,MR1170831,MR1105467,MR318007,MR1275613}.

			\section{Preliminaries}  
			\subsection{Adjacency, Laplacian, and signless Laplacian of a graph.}
			Let $\Gamma$ be a graph with $V(\G)$ as the vertex set, and $E(\G)$ is the edge set of $\G$. For all $v\in V(\G)$, the \emph{neighbourhood of} $v$ is $N_{\G}(v)=\{u\in V(\G):\{u,v\}\in E(\G)\}$. The set of all functions from $V(\G)$ to $\mathbb{R}$ is denoted by $\mathbb{R}^{V(\G)}$. The \textit{adjacency } $A_{\Gamma}:\mathbb{R}^{V(\G)}\to \mathbb{R}^{V(\G)}$, the \textit{signless Laplacian} $Q_{\Gamma}:\mathbb{R}^{V(\G)}\to \mathbb{R}^{V(\G)}$, \textit{the Laplacian} $L_{\Gamma}:\mathbb{R}^{V(\G)}\to \mathbb{R}^{V(\G)}$ associated with $\Gamma$ are defined by 
			$$(A_\Gamma x)(v) = \sum\limits_{u(\ne v):\{u,v\}\in E(\G)}x(u), $$ $$(Q_{\Gamma}x)(v)=\sum\limits_{u(\ne v):\{u,v\}\in E(\G)}(x(v)+x(u)),$$ and $$(L_{\Gamma}x)(v)=\sum\limits_{u(\ne v):\{u,v\}\in E(\G)}(x(v)-x(u)),$$ respectively, for $x\in \mathbb{R}^{V(\G)}$, and $v\in V(\G)$. For any $x\in \mathbb{R}^{V(\G)}$, the support of $x$ is $supp(x)=\{v\in V(\G):x(v)\ne 0\}$. For each $U\subseteq V(\G)$, we define $\chi_{U}\in \mathbb{R}^{V(\G)}$ as 
			$$\chi_{U}(v)=\begin{cases}
				1&\text{~if~}v\in U,\\
				0&\text{~otherwise,}
			\end{cases}$$
			and we define $T_U=\{x\in\mathbb{R}^{V(\G)}:supp(x)\subseteq U, \sum\limits_{u\in U}x(u)=0\}$.
			\begin{lem}\label{lemtu}
				Let $\G$ be a graph.
				For any $U=\{u_0,\ldots,u_k\}\subseteq V(\G)$, $T_U$ is a subspace of $\mathbb{R}^{V(\G)}$ of dimension $|U|-1$  and $\mathcal{Y}_U=\{y_i=\chi_{\{u_0\}}-\chi_{\{u_i\}}:i=1,\ldots,k\}$ is a basis of $T_U$.
			\end{lem}
			\begin{proof}
				Let $x,y\in T_U$. If $v\in V(\G)\setminus U$ then $(c_1x+c_2y)(v)=0$ for all $c_1,c_2\in\mathbb{R}$, and $ \sum\limits_{u\in U}(c_1x+c_2y)(u)=0$. Thus, $c_1x+c_2y\in T_U$ for all $c_1,c_2\in\mathbb{R}$. Therefore, $T_U$ is a subspace of $\mathbb{R}^{V(\G)}$.
				
				For $c_1,\ldots,c_k\in\mathbb{R}$, if $\sum\limits_{i=1}^kc_iy_i=0$ then $c_i=0$ for all $i=1,\ldots,k$. Thus, the set $\mathcal{Y}_U$ is linearly independent. Since $x=\sum\limits_{i=1}^kx(i)y_i$ for all $x\in T_U$, the linearly independent set $\mathcal{Y}_U$ span $T_U$. Therefore, $\mathcal{Y}_U$ is a basis of $T_U$ and the dimension of $T_U$ is $|U|-1$.  
			\end{proof}

			For any linear operator $M:V\to V$ on a finite-dimensional vector space $V$, we denote the set of all eigenvalues of $M$ as $\sigma(M)$. For any $\lambda\in \sigma(M)$, we denote the eigenspace of $\lambda$ as $ \Omega_{\lambda}(M)$.
			\subsection{Twin-neighbours and non-adjacent twin}
			\begin{defn} For a graph $\Gamma$, we refer to $u,v\in V(\G)$ as \emph{twin-neighbours} in $\Gamma$  if $\{u,v\}\in E(\G)$, and  $N_{\G}(v)\setminus \{u\}=N_{\G}(u)\setminus \{v\}$. We call $u,v\in V(\G)$ as \emph{non-adjacent twin} in $\Gamma$  if $\{u,v\}\notin E(\G)$, and  $N_{\G}(v)=N_{\G}(u)$.
			\end{defn}
			For example, in \Cref{fig:twin}, $u$, $v$  are twin-neighbours, and $u'$, $v'$  are non-adjacent twin.
			\begin{figure}[ht]
				\centering
				\begin{tikzpicture}[scale=0.6]
					
					\node [style=new style 0] (0) at (-12, 0) {$u$};
					\node [style=new style 0] (1) at (-8, 0) {$v$};
					\node [style=new style 0] (2) at (-10, 3) {};
					\node [style=new style 0] (3) at (-10, 4) {};
					\node [style=new style 0] (4) at (-10, 6) {};
					\node [style=none] (5) at (-10, 5.25) {$\vdots$};
					\node [style=none] (6) at (-10, 4.75) {$\vdots$};
					\node [style=new style 0] (7) at (0, 0) {$u'$};
					\node [style=new style 0] (8) at (4, 0) {$v'$};
					\node [style=new style 0] (9) at (2, 3) {};
					\node [style=new style 0] (10) at (2, 4) {};
					\node [style=new style 0] (11) at (2, 6) {};
					\node [style=none] (12) at (2, 5.25) {$\vdots$};
					\node [style=none] (13) at (2, 4.75) {$\vdots$};
					\node [style=none] (14) at (-10, -1.75) {Twin-neighbours};
					\node [style=none] (15) at (2, -1.75) {Non-adjacent twin};
					
					\draw (0) to (1);
					\draw (2) to (0);
					\draw (2) to (1);
					\draw (3) to (0);
					\draw (3) to (1);
					\draw (4) to (0);
					\draw (4) to (1);
					%\draw (7) to (8);
					\draw (9) to (7);
					\draw (9) to (8);
					\draw (10) to (7);
					\draw (10) to (8);
					\draw (11) to (7);
					\draw (11) to (8);
				\end{tikzpicture}
				\caption{Twin-neighbours, and non-adjacent twin.}
				\label{fig:twin}
			\end{figure}
			% The notion of \emph{twin-neighbours } leads us to the following equivalence relation on the vertex set.

			% $$\mathcal{R}=\{\{u,v\}\in V^2: u\text{~and~}v\text{~are~twin-neighbours, and~}\{u,v\}\in E\}$$
			
			\begin{lem}[]{\rm (\cite[Theorem 5.10.]{MR1352837})}\label{nbd-eig}
				Let $\Gamma$ be a graph. If $u,v\in V(\G)$ are twin-neighbours, then $-1$, $|N(v)|-1$, and $|N(v)|+1$ are eigenvalues of $A_\Gamma$, $Q_\Gamma$, and $L_\Gamma$, respectively, and the corresponding eigenvector for each operator is $\chi_{\{u\}}-\chi_{\{v\}}$.
			\end{lem}
			\begin{proof}
				\sloppy Suppose that $y=\chi_{\{u\}}-\chi_{\{v\}}$. Since $\sum\limits_{w(\ne u):\{u,w\}\in E}y(w)=-1$, $\sum\limits_{w(\ne v):\{v,w\}\in E}y(w)=1$, and for all $u^\prime\notin\{u,v\}$, the sum $\sum\limits_{w(\ne u^\prime):\{u^\prime,w\}\in E}y(w)=0$, and we have $Ay=-y$.
				%$$ Ay(u^\prime )=\begin{cases} \sum\limits_{w(\ne u):\{u,w\}\in E}y(w)&\text{~if~}u^\prime=v \\ \sum\limits_{w(\ne v):\{v,w\}\in E}y(w)&\text{~if~}u^\prime=u \\  \sum\limits_{w(\ne u^\prime):\{u^\prime,w\}\in E}y(w)&\text{~otherwise.}\end{cases}=\begin{cases}  -1&\text{~if~}u^\prime=v \\   1&\text{~if~}u^\prime=u \\   0&\text{~otherwise.}  \end{cases}=-y(u^\prime)$$
				\sloppy For signless Laplacian, we have $\sum\limits_{w(\ne u):\{u,w\}\in E(\G)}(y(u)+y(w))=|N(u)|-1$,  $\sum\limits_{w(\ne v):\{v,w\}\in E(\G)}(y(u)+y(w))=-|N(v)|+1$,  and for all $u^\prime\notin\{u,v\}$, $\sum\limits_{w(\ne u^\prime):\{u^\prime,w\}\in E(\G)}(y(u^\prime)+y(w))=0$. Therefore, $ Q_{\Gamma}y=(|N(u)|-1)y$.
				
				Similarly, $\sum\limits_{w(\ne u):\{u,w\}\in E(\G)}(y(u)-y(w))=|N(u)|+1$,  $\sum\limits_{w(\ne v):\{v,w\}\in E(\G)}(y(u)-y(w))=-(|N(v)|+1)$, and for all $u^\prime\notin\{u,v\}$, $\sum\limits_{w(\ne u^\prime):\{u^\prime,w\}\in E(\G)}(y(u^\prime)-y(w))=0$. Thus, $ L_{\Gamma}y=(|N(u)|+1)y$. 
			\end{proof}
			% \begin{lem}
				%     Let $\Gamma$ be a graph. If $u,v\in V(\G) $ is such that $-1$ is an eigenvalue of $A_{\Gamma}$ with eigenvector  $\chi_{\{u\}}-\chi_{\{v\}}$ then $u,v$ are unit-neighbours in $\Gamma$.
				% \end{lem}
			% \begin{proof}
				%   Let $y=\chi_{\{u\}}-\chi_{\{v\}}$. Since $A_{\Gamma}y=-y$, for all $w\notin \{u,v\}$, we have $0=-y(w)=(A_{\Gamma}y)(w)$. Thus, $$\sum\limits_{w^\prime(\ne w):\{w,w^\prime\}\in E}y(w^\prime )=0.$$ 
				%   Therefore, for all $w\notin \{u,v\}$, either $w\in N(u)\cap N(v)$ or  $w\in V\setminus (N(u)\cup N(v)) $. Thus,  $N(v)\setminus \{u\}=N(u)\setminus \{v\}$.
				%   since $(A_{\Gamma}y)(u)=-y(u)=-1$, thus, $$\sum\limits_{w(\ne u):\{u,w\}\in E}y(w)=-1.$$
				%   Therefore, $ v\in N(u)$ and $\{u,v\}\in E(\G)$ and the result follows.  
				% \end{proof}
			\section{The spectra of commuting graphs}\label{spectra}
			Let $G$ be a group. The center of $G$ is denoted by
			$ Z(G)=\{g\in G:gu=ug \text{~for all~}u\in G\}$
			and for $u\in G$, the \emph{centralizer of $u$} is denoted by
			$$C(u)=\{g\in G:gu=ug\}.$$
			
			If $G$ is a commutative group then $\mathcal{C}_G$ is a complete graph and thus, using Lemma \ref{nbd-eig}, the eigenvalues of $A_{\mathcal{C}_G}$ are $ |G|-1$ with multiplicity $1$ and $-1$ with multiplicity $ |G|-1$. The eigenvalues of $Q_{\mathcal{C}_G}$ are $ 2(|G|-1)$ with multiplicity $1$ and $ |G|-2$ with multiplicity $|G|-1$.
			The eigenvalues of $L_{\mathcal{C}_G}$ are $ 0$ with multiplicity $1$ and $ |G|$ with multiplicity $|G|-1$. Now we consider the case when $G$ is non-commutative. 
			%\subsection{Equivalence relation of having same centralizer}
			
			For any group $G$, if $u,v(\ne u)\in G$ are such that $C(u)=C(v)$, then $u,v$ are twin-neighbours in $\com{G}$. This fact makes the following equivalence relation interesting.
			\begin{defn}\label{equidef}
				Let $G$ be a group. We define a relation,
				$$\mathcal{R}_G=\{\{u,v\}\in G\times G:C(u)=C(v)\}$$
				on $G$.
			\end{defn}
			Evidently, $\mathcal{R}_G$ is an equivalence relation on $G$. For any $u\in G$, we denote the $\mathcal{R}_G$-equivalence class of $u$ as $\mathfrak{F}_u$, and the collection of all the $\mathcal{R}_G$-equivalence classes in $G$ as $\mathfrak{G} $. Suppose that $\mathfrak{G} =\{\mf{F}_0,\ldots \mf{F}_m\} $. One of these equivalence classes is the centre of $G$. We assume $\mf{F}_0=Z(G)$. For any $u\in G$, if $u\in \mf{F}_i$, then $\mf{F}_i=\mf{F}_u$. For $u,v\in G$ with $ \mf{F}_u\ne\mf{F}_v$, if $uv=vu$ then $u^\prime v^\prime=v^\prime u^\prime$ for all $u^\prime\in\mf{F}_u, v^\prime\in\mf{F}_v$, in that case we say two equivalence class $\mf{F}_u$, and $\mf{F}_v$ are \emph{adjacent}. Therefore, for all $v\in C(u)$, either $v\in \mf{F}_u$ or $\mf{F}_u$ and $\mf{F}_v$ are adjacent. For any group $G$, the next result provides some relations of $C(u)$, and $\mf{F}_u$ with spectra of $\com{G}$ for all $u\in G$. 
			
			\begin{prop}\label{cu}
				Let $G$ be a group, and $\Gamma=\com{G}$. If $u\in G $, then  $-1$, $|C(u)|-2$, and $|C(u)|$ are eigenvalues of $A_\Gamma$, $Q_\Gamma$, and $L_\Gamma$, respectively with multiplicity $|\mathfrak{F}_u|-1$ in each cases.
			\end{prop} 
			\begin{proof} %\textcolor{blue}{You are right. I had some confusion which is clear now.}
				Since for all $u\in G$, we have  $N(u)=C(u)\setminus \{u\}$, the result follows from \Cref{nbd-eig}.
			\end{proof}
			
			Using \Cref{cu}, we can find $|G|- \d\limits_{i=0}^m(|\mf{F}_i|-1)$ eigenvalues of each of $A_\Gamma$, $Q_\Gamma$, and $L_\Gamma$. Now, for the remaining $|\mathfrak{G}| $ eigenvalues, we use the graph $ \Gamma_{G/\mathfrak{G}}$ with vertex set $V(\Gamma_{G/\mathfrak{G}})= \mathfrak{G}$, and $$E(\Gamma_{G/\mathfrak{G}})=\{\{\mf{F}_u,\mf{F}_v\}: \mf{F}_u,\text{~and~}\mf{F}_v \text{~are adjacent class}\}.$$
			We enumerate $\mf{G}$ as $\mf{G}=\{\mf{F}_0(=Z(G)),\mf{F}_1,\ldots,\mf{F}_m\}$, and if $\mf{F}_i$, and $\mf{F}_j$ are adjacent in $\Gamma_{G/\mathfrak{G}}$, we write $i\sim j$. For any function $x:\mf{G}\to \mathbb{R}$, we define the \emph{blow up} of $x$ is the function $\ov{x}:G\to \mathbb{R}$ such that $\ov{x}(v)=x(\mf{F}_v)$ for all $v\in G$.
			
			\begin{ex}\label{s4initial}\rm
				We enumerate the symmetric group of four elements  $\mf{S}_4$ as $1= e$, the identity element, $2= (3,4) $, $3=  (2,3) $, $ 4= (2,3,4)$, $5=(2, 4, 3)$, $ 6=(2, 4)$, $7=(1, 2)$, $8=(1, 2)(3, 4))$, $ 9=(1, 2, 3)$, $10=(1, 2, 3, 4)$, $11=(1, 2, 4, 3)$, $12=(1, 2, 4)$, $13=(1, 3, 2)$, $ 14=(1, 3, 4, 2)$, $ 15=(1, 3)$, $ 16=(1, 3, 4)(1, 3)$, $ 17=(2, 4)$, $ 18=(1, 3, 2, 4)$, $ 19=(1, 4, 3, 2)$, $20=(1, 4, 2)$, $21=(1, 4, 3)$, $22=(1, 4)$, $ 23=(1, 4, 2, 3)$, $ 24=(1, 4)(2, 3)$.

				\begin{figure}[ht]
					\centering
					\includegraphics[scale=0.25]{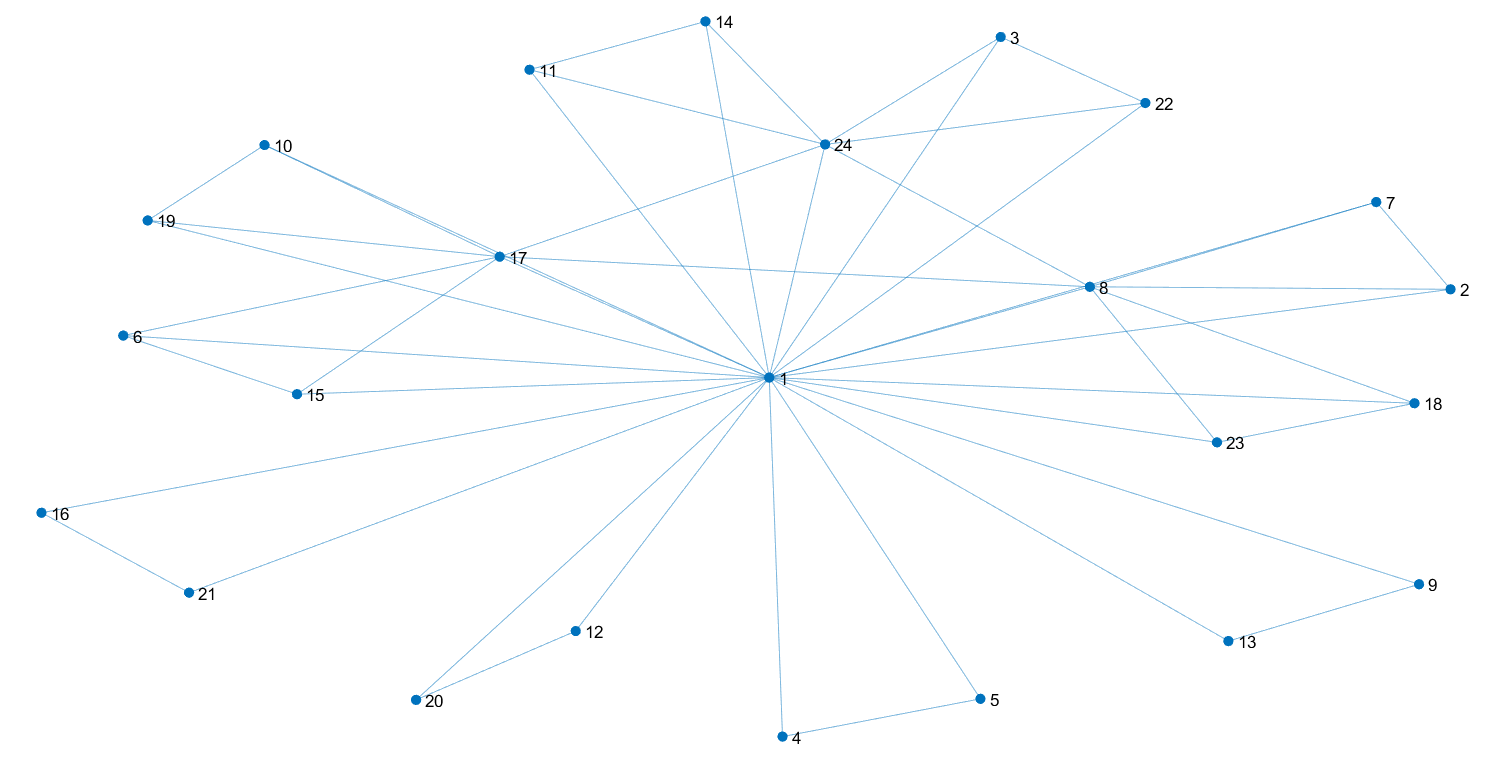}
					\caption{$\com{\mf{S}_4}$}
					\label{fig:S-4}
				\end{figure}
				For $\com{\mf{S}_4}$, see \Cref{fig:S-4}. For this group, 
				$$\mf{G}=\{\mf{F}_1, \mf{F}_{17}, \mf{F}_{24}, \mf{F}_{8}, \mf{F}_{16},\mf{F}_{10}, \mf{F}_{11}, \mf{F}_{3}, \mf{F}_{2}, \mf{F}_{18}, \mf{F}_{9}, \mf{F}_{4}, \mf{F}_{12}, \mf{F}_{16}\},$$
				
				where $\mf{F}_1=Z(\mf{S}_4)=\{1\}$, $\mf{F}_{17}=\{17\}$, $\mf{F}_{24}=\{24\}$,$ \mf{F}_{8}=\{8\}$, $\mf{F}_{16}=\{16,15\}$, $ \mf{F}_{10}=\{10,19\}$, $\mf{F}_{11}=\{11,14\}$, $\mf{F}_{3}=\{3,22\}$, $\mf{F}_{2}=\{2,7\}$, $\mf{F}_{18}=\{18,23\} $, $\mf{F}_{9}=\{9,13\}$, $\mf{F}_{4}=\{4,5\}$, $\mf{F}_{12}=\{12,20\}$, $\mf{F}_{16}=\{16,21\}$. The graph $\Gamma_{\mf{S}_4/\mathfrak{G}}$ is given in \Cref{fig:contraction}.
			\end{ex}
			\begin{figure}[ht]
				\centering
				\begin{tikzpicture}[scale=0.6]
					
					\node [style=new style 0] (1) at (0, 2) {\tiny $\mf F_{24}$};
					\node [style=new style 0] (2) at (2, 4) {\tiny $\mf F_{3}$};
					\node [style=new style 0] (3) at (-2, 4) {\tiny $\mf F_{11}$};
					\node [style=new style 0] (5) at (-2.25, 0.75) {\tiny $\mf F_{17}$};
					\node [style=new style 0] (6) at (-3.75, 3.25) {\tiny $\mf F_{10}$};
					\node [style=new style 0] (7) at (-4.5, -0.5) {\tiny $\mf F_{6}$};
					\node [style=new style 0] (8) at (0, 0) {\tiny $\mf F_1$};
					\node [style=new style 0] (9) at (2.25, 0.75) {\tiny $\mf F_8$};
					\node [style=new style 0] (10) at (4.5, -0.5) {\tiny $\mf F_{18}$};
					\node [style=new style 0] (11) at (3.5, 3.25) {\tiny $\mf F_{7}$};
					\node [style=new style 0] (12) at (-3.25, -3) {\tiny $\mf F_{16}$};
					\node [style=new style 0] (13) at (-1.25, -3.75) {\tiny $\mf F_{12}$};
					\node [style=new style 0] (14) at (1.25, -3.75) {\tiny $\mf F_{4}$};
					\node [style=new style 0] (15) at (3.25, -3) {\tiny $\mf F_{9}$};
					
					\draw (1) to (2);
					\draw (1) to (3);
					\draw (5) to (6);
					\draw (5) to (7);
					\draw (8) to (9);
					\draw (9) to (10);
					\draw (9) to (11);
					\draw (11) to (8);
					\draw (8) to (10);
					\draw (8) to (1);
					\draw (3) to (8);
					\draw (8) to (2);
					\draw (8) to (5);
					\draw (7) to (8);
					\draw (8) to (6);
					\draw (8) to (13);
					\draw (8) to (14);
					\draw (8) to (15);
					\draw (8) to (12);
					\draw (1) to (5);
					\draw (5) to (9);
					\draw (9) to (1);
				\end{tikzpicture}
				
				\caption{$\Gamma_{\mf{S}_4/\mathfrak{G}}$}
				\label{fig:contraction}
			\end{figure}

			%\subsection{Condition}
			
			%\par{\textbf{Group information encoded in the Laplacian spectrum of its commuting graph.}} 
			
			\subsection{Laplacian Spectrum of a Commuting Graph.}Now, we explore the Laplacian spectra of the commuting graph associated with a group to find the group information encoded in the spectrum. Before going into the next result, we need to introduce a matrix $[L_{\com{G}}/\mf{G}]=(l_{ij})_{\mf{F}_i,\mf{F}_j\in\mf{G }}$ associated with $\Gamma_{G/\mathfrak{G}}$ defined as 
			$$l_{ij}=\begin{cases}
				{-} |\mf{F}_j|&\text{~if~}i\ne j \text{~and~} \mf{F}_i \text{ is adjacent to }\mf{F}_j,\\
				\phantom{-}0&\text{~if~}i\ne j \text{~and~} \mf{F}_i \text{ is not adjacent to }\mf{F}_j,\\
				-\sum\limits_{j(\ne i)=0}^ml_{ij} &\text{~if~}i=j.
			\end{cases} $$
			
			\begin{thm}\label{l-contraction}
				% For any group $G$, $\sigma([L_{\com{G}}/\mf{G}])\subseteq \sigma(L_{\com{G}})$. If $x\in \Omega_{\lambda}([L_{\com{G}}/\mf{G}])$ for any $\lambda\in \sigma([L_{\com{G}}/\mf{G}])$, then the blow up $\ov{x}\in \Omega_{\lambda}(L_{\com{G}})$.

				Let $G$ be a group. 
				If $\lambda$ is an eigenvalue of the matrix $[L_{\com{G}}/\mf{G}]$ with an eigenvector $x:\mf{G}\to\mathbb{R}$, then
				% $=(l_{ij})_{\mf{F}_i,\mf{F}_j\in\mf{G }}$ of order $|\mf{G}|$, defined by
				% $$l_{ij}=\begin{cases}
					%   {-} |\mf{F}_j|&\text{~if~}i\ne j \text{~and~} \mf{F}_i \text{ is adjacent to }\mf{F}_j,\\
					%      \phantom{-}0&\text{~if~}i\ne j \text{~and~} \mf{F}_i \text{ is not adjacent to }\mf{F}_j,\\
					%    -\sum\limits_{j(\ne i)=0}^ml_{ij} &\text{~if~}i=j,
					% \end{cases} $$
				$\lambda$ is also an eigenvalue of $L_{\com{G}}$ with an eigenvector $\ov{x}:{G}\to\mathbb{R}$ defined by $\ov{x}(v)=x(\mf{F}_v)$ for all $v\in V(\G)$.
			\end{thm}
			\begin{proof}
				For any $v\in V(\G)$, suppose that $\mf{F}_v=\mf{F}_i$. Thus,
				\begin{align*}
					(L_{\G}\ov{x}(v))
					&=\d\limits_{u\in C(v)}(\ov{x}(v)-\ov{x}(u))
					=\d\limits_{u\in C(v)}({x}(\mf{F}_v)-{x}(\mf{F}_u))\\
					&=\left(-\d\limits_{j(\ne i)=0}^ml_{ij}\right)x(\mf{F}_i)-\d\limits_{j(\ne i)=0}^m|\mf{F}_j|x(\mf{F}_j)\\
					&=([L_\G/\mf{G}]x)(\mf{F}_i)=\lambda x(\mf{F}_i)=\lambda \ov{x}(v). 
				\end{align*}
				This completes the proof.   
			\end{proof}
			\begin{thm}\label{z(g)-contract}
				For any group $G$, one eigenvalue of $ L_{\com{G}}$ is $|G|$ with eigenvector $x:G\to \mathbb{R}$, defined by $x_{\small Z(G)}=(|G|-|Z(G)|)\chi_{Z(G)}-|Z(G)|\d\limits_{\mf{F}_j(\ne Z(G))\in\mf{G}}\chi_{\mf{F}_j}$.
			\end{thm}
			\begin{proof}
				We claim $|G|$ is an eigenvalue of $[L_\G/\mf{G}]$ and the corresponding eigenvector is 
				$y:\mf{G}\to \mathbb{R} $, defined by
				$$y(\mf{F}_u)=
				\begin{cases}
					\phantom{-}  |G|-|Z(G)| &\text{~if~} \mf{F}_u=Z(G),\\
					- |Z(G)| &\text{~otherwise.}
				\end{cases}$$
				Let $\mf{G}=\{\mf{F}_0,\ldots,\mf{F}_m\}$.
				The following observations about $ ([L_\G/\mf{G}]y)$ establish the above claim.
				\begin{enumerate}
					\item If $\mf{F}_i=Z(G)$, then 
					\begin{align*}
						([L_\G/\mf{G}]y)(\mf{F}_i)&=(|G|-|Z(G)|)(|G|-|Z(G)|)+|Z(G)|(|G|-|Z(G)|)\\
						&=|G|(|G|-|Z(G)|).
					\end{align*}
					\item If $\mf{F}_i\ne Z(G)$, then 
					\begin{align*}
						([L_\G/\mf{G}]y)(\mf{F}_i)&=-|Z(G)|(|G|-|Z(G)|)-|Z(G)||Z(G)|\\
						&=-|G||Z(G)|.
					\end{align*}
					Since $x_{\small Z(G)}=\ov{y}$, this completes the proof.
				\end{enumerate}
			\end{proof}
			For any graph $\G$, the \emph{complement of $\G$}, is a graph $\ov{\G}$  with $V(\ov{\G})=V(\G)$, and $E(\G)=\{\{u,v\}\subseteq V(\ov{\G}):\{u,v\}\notin E(G)\}$. Recall that if $\mf{F}_i$, and $\mf{F}_j$ are adjacent in $\Gamma_{G/\mathfrak{G}}$ then we write $i\sim j$. Thus, if $\mf{F}_i$, and $\mf{F}_j$ are adjacent in $\ov{\Gamma_{G/\mathfrak{G}}}$ then we write it as $i\not\sim j$. For a commutative group $G$, the multiplicity of the eigenvalue $|G|$ of $L_{\com{G}}$ is exactly $|G|-1$.
			The next result provides more information about the eigenvalue $|G|$ and its eigenspace for a non-commutative group $G$.
			\begin{thm}
				Let $G$ be a non-commutative group.
				The multiplicity of the eigenvalue $|G|$ of $L_{\com{G}}$ is $|Z(G)|+r-2$, where $r$ is the number of components in the complement of the graph $\G_{G/\mf{G}}$.
			\end{thm}
			\begin{proof}
				To prove this result, it is enough to prove the multiplicity of the eigenvalue $|G|$ of $ [L_\G/\mf{G}]$ is $r-1$.
				Let $x:\mf{G}\to\mathbb{R}$ is an eigenvector associated with eigenvalue $|G|$ of $ [L_\G/\mf{G}]$. In that case, by \Cref{l-contraction}, $\ov{x}:G\to\mathbb{R}$ is a eigenvector associated with eigenvalue $|G|$ of $L_{\com{G}}$. Since $0$ is an eigenvalue of $L_{\com{G}}$ with eigenvector $\mathbf{1}:G\to\mathbb{R}$ defined by $\mathbf{1}(v)=1$ for all $v\in G$, we have $\ov{x}$ is orthogonal to $\mathbf{1}$. Thus, if $\mf{G}=\{\mf{F}_0=Z(G),\mf{F}_1,\ldots,\mf{F}_m\}$,
				\be\label{zg-1}
				\d\limits_{i=0}^m|\mf{F}_i|x(\mf{F}_i)=0.
				\ee
				Now since $ [L_\G/\mf{G}]x=|G|x$, for all $i=0,\ldots,m$,
				\begin{align}
					\label{zg-2}&\phantom{\implies}x(\mf{F}_i)\d\limits_{j(\ne i):i\sim j}|\mf{F}_j|-  \d\limits_{j(\ne i):i\sim j}|\mf{F}_j|x(\mf{F}_j)=|G|x(\mf{F}_i)\notag\\
					&\implies -  \d\limits_{j(\ne i):i\sim j}|\mf{F}_j|x(\mf{F}_j)=\left(|G|-\d\limits_{j(\ne i):i\sim j}|\mf{F}_j|\right)x(\mf{F}_i).
				\end{align}
				Thus, using \Cref{zg-1}, and the fact $|G|=\d\limits_{i=0}^m|\mf{F}_i|$,
				\begin{align}
					&\phantom{\implies} x(\mf{F}_i)|\mf{F}_i|+\d\limits_{j(\ne i):i\not\sim j}|\mf{F}_j|x(\mf{F}_j)=|\mf{F}_i|x(\mf{F}_i)+x(\mf{F}_i)\d\limits_{j(\ne i):i\sim j}|\mf{F}_i|\notag\\
					&\implies \d\limits_{j(\ne i):i\not\sim j}(x(\mf{F}_j)-x(\mf{F}_i))|\mf{F}_j|=0\notag\\
					& \implies M_Hx=0,
				\end{align}
				where $H=\ov{\G_{G/\mf{G}}}$, and $M_H=\left(m_{ij}\right)_{\mf{F}_i,\mf{F}_j\in\mf{G}}$ is defined by
				$$ m_{ij}=
				\begin{cases}
					-|\mf{F}_j|&\text{~if~}i\not\sim j \text{~and~}i\ne j,\\
					\phantom{-} 0&\text{~if~}i\sim j \text{~and~}i\ne j,\\ 
					-\d\limits_{j(\ne i)=1}^mm_{ij} &\text{~otherwise.}
				\end{cases}$$
				Therefore, $x$ is an eigenvector associated with eigenvalue $|G|$ of $ [L_\G/\mf{G}]$ if $x$ belongs to the kernel of $M_H$, and $x$ satisfy \Cref{zg-1}.
				
				Let $D$ be a diagonal matrix of order $|\mf{G}|$ such that $D_{ii}=|\mf{F}_i|$ for all $i=0,\ldots,m$. Since, $M_H$, and $DM_H$ have the same kernel, and $DM_H$ is a weighted Laplacian matrix of the graph $H$, with weight of an edge $\{\mf{F}_i,\mf{F}_j\}(\in E(K) )$ is $|\mf{F}_i||\mf{F}_j|$. Thus, the dimension of the kernel of $M_H$ is $r$, where $r$ is the number of connected components in $H$.
				
				Let $\mathbf{1}_{\mf{G}}:\mf{G}\to \mathbb{R}$ be defined by $\mathbf{1}_{\mf{G}}(\mf{F}_i) =1$, for all $\mf{F}_i\in \mf{G}$. Since $\mathbf{1}_{\mf{G}}$ is a vector belongs to the kernel of $M_H$, does not satisfy \Cref{zg-1}, and all the elements of the kernel, that is linearly independent to $\mathbf{1}_{\mf{G}}$, satisfy \Cref{zg-1}, the multiplicity of eigenvalue $|G|$ of $ [L_\G/\mf{G}]$ is $r-1$.
			\end{proof}
			Since $\mf{F}_0=Z(G)$ is an isolated vertex in the complement of the graph $\G_{G/\mf{G}}$, the value of $r$ is at least $2$, and thus, the multiplicity of the eigenvalue $|G|$ of $L_{\com{G}}$ is at least $|Z(G)|$. Since $C(u)=G$ for all $u\in Z(G)$, \Cref{cu}, and \Cref{l-contraction} lead us to the following result.
			\begin{thm}[Complete spectrum of $L_{\com{G}}$]\label{complete-l}
				For a finite group $G$, the eigenvalues of $L_{\com{G}}$ are listed below.
				\begin{enumerate}
					\item If $u\in G\setminus Z(G)$, then $|C(u)|$ is an eigenvalue of multiplicity $|\mf{F}_u|-1$.
					\item The order of the group, $|G|$ is an eigenvalue.
					If $G$ is non-commutative, then the multiplicity of $|G|$ is at least $|Z(G)| $. If $G$ is commutative, then the multiplicity of $|G|$ is $|Z(G)|-1$.
					
					% If $G$ is non-commutative $|G|$ is an eigenvalue of multiplicity at least $|Z(G)| $. If $G$ is commutative then the multiplicity of $|G|$ is $|Z(G)|-1$.
					\item The remaining eigenvalues are the eigenvalues of $[L_{\com{G}}/\mf{G}] $, that is
					$\sigma([L_{\com{G}}/\mf{G}])\subseteq \sigma(L_{\com{G}})$. If $x\in \Omega_{\lambda}([L_{\com{G}}/\mf{G}])$ for any $\lambda\in \sigma([L_{\com{G}}/\mf{G}])$, then the blow up $\ov{x}\in \Omega_{\lambda}(L_{\com{G}})$.
					
					% For every eigenvalue $\lambda$ with an eigenvector $x:\mf{G}\to\mathbb{R}$ (with $x\ne x_{\small Z(G)}$) of the matrix $[L_\G/\mf{G}]=(l_{ij})_{\mf{F}_i,\mf{F}_j\in\mf{G }}$ for
					% $$l_{ij}=\begin{cases}
						%   {-} |\mf{F}_j|&\text{~if~}i\ne j \text{~and~} \mf{F}_i \text{ is adjacent to }\mf{F}_j,\\
						%      \phantom{-}0&\text{~if~}i\ne j\text{~and~} \mf{F}_i \text{ is not adjacent to }\mf{F}_j,\\
						%     -\sum\limits_{j(\ne i)=0}^ml_{ij} &\text{~if~}i=j,
						% \end{cases} $$
					% $\lambda$ is an eigenvalue of $L_\Gamma$ with eigenvector $\ov{x}:{G}\to\mathbb{R}$, defined by $\ov{x}(v)=x(\mf{F}_v)$ for all $v\in G$.
				\end{enumerate}
			\end{thm}
			\begin{proof} If $G$ is commutative then $G\setminus Z(G)=\emptyset$, and thus part (1) follows trivially. In that case, $\com{G}$ is a complete graph so $|G|$ is an eigenvalue of multiplicity $|Z(G)|-1$. Since $\mf{G}$ is a singleton set, only eigenvalue of $[L_\G/\mf{G}]$ is $0$, which is also an eigenvalue of $L_{\com{G}}.$ 
				
				If $G$ is  non-commutative, then the
				part (1) of the result follows from \Cref{cu}, and \Cref{lemtu}. For part (2), if $Z(G)=\{u_0,\ldots,u_k\}$ then by \Cref{cu}, $|G|$ is a an eigenvalue of  $L_{\com{G}}$ with eigenvector $y_i=\chi_{\{u_i\}}-\chi_{\{u_0\}}$, for all $i=1,\ldots,k$. Since, using \Cref{z(g)-contract}, we have $\{x_{\small Z(G)},y_1,\ldots,y_k\}$ are  $|Z(G)|$ many linearly independent eigenvectors corresponding to $|G|$, part (2) of the result follows. Part (3) directly follows from \Cref{l-contraction}.
				
				Since from \Cref{cu}, \Cref{l-contraction}, and \Cref{z(g)-contract}, it is clear that the $|G|$ number of eigenvectors corresponding to the listed eigenvalues are linearly independent, this list provides the complete spectrum of  $L_{\com{G}}$.
			\end{proof}
			%We denote the spectrum of a matrix $L$ as $\sigma(L)$.
			\begin{ex}[$\mf{S}_4$]\rm
				We have provided the commuting graph of $\mf{S}_4$ in \Cref{s4initial} (\Cref{fig:S-4}).  We enlist $\sigma(L_{\com{\mf{S}_4}})$ below by using \Cref{complete-l}.
				\begin{enumerate}[leftmargin=*]
					\item Since $C(18)=\{1,8, 18,23\}$, and $\mf{F}_{18}=\{18,23\}$, we have $|C(18)|=4$ is an eigen value of multiplicity at least $1$. Similarly, 
					\begin{enumerate}
						\item  $C(7)=\{1,2,7, 8\}$, and $\mf{F}_{7}=\{2,7\}$; 
						\item $C(3)=\{1,3,22,24\}$, and $\mf{F}_{3}=\{3,22\}$;
						\item $C(11)=\{1,11,14,24\}$, and $\mf{F}_{11}=\{11,14\}$;
						\item $C(10)=\{1,10,19,17\}$, and $\mf{F}_{10}=\{10,19\}$;
						\item $C(6)=\{1,6,15,17\}$, and $\mf{F}_{6}=\{6,15\}$.
					\end{enumerate}
					Thus, the multiplicity of $4$ is at least $6$.
					\item Similarly, because of $C(16), C(12), C(4),$ and $ C(9)$, we have $3$ is an eigenvalue of $L_{\com{\mf{S}_4}}$ with multiplicity at least $4$.
					\item The remaining $14$ eigenvalues are the eigenvalues of the matrix $[L_{\com{\mf{S}_4}}/\mf{G}] $ of order $14$. One of them is $|\mf{S}_4|=24$ with multiplicity $|Z(G)|=1$. Another one is $0$ with multiplicity $1$.
				\end{enumerate}
			\end{ex}
			\begin{ex}[Finite pseudoreflection Group]\rm\label{pseu}
				A \emph{pseudoreflection} on $\C^d$ is a linear homomorphism $\sigma: \C^d \rightarrow \C^d$ such that $\sigma$ has finite order in $GL(d,\mb C)$ and the rank of $id - \sigma$ is 1. A group generated by pseudoreflections is called a pseudoreflection group. If the order of a pseudoreflection is $2,$ we call it a reflection. In a pseudorflection group, a Coxeter element is a product of all reflections of the group. In particular, if a pseudoreflection group is a Coxeter group, a product of all generators is a Coxeter element, see \cite[29-1, p. 299]{MR1838580}. By \cite[Theorem B, p. 300]{MR1838580}, if a finite reflection group is irreducible, then the order of any Coxeter element is $h,$ where $h$ is the Coxeter number of the group. Two Coxeter elements $u$ and $v$ are conjugate, so $|C(u)| = |C(v)|.$  Moreover, for every  Coxeter element $u$, $C(u)$ is cyclic and $|C(u)|=h$.
				\begin{enumerate}[leftmargin=*]
					\item 
					Let $G$ be a finite irreducible pseudorflection group with Coxeter number $h>2.$ From \Cref{complete-l}, $h$ is an eigenvalue of the Laplacian matrix of the commuting graph $\mathcal C_G$ associated to the group $G.$ The multiplicity of the eigenvalue $h$ is at least $k(h-|Z(G)|-1),$ where the constant $k = |\{[u]: u \text{~is a Coxeter element}\}|$ and $[u]=\{v: v \text{~is Coxeter and~} C(u) =C(v) \}.$
					\item \label{dihed}
					The dihedral group $D_{2m}= \langle x,y: x^m=y^2=id \text{~and~} yx=x^{m-1}y\rangle$ is a finite pseudoreflection group. The group $D_{2m}$ has Coxeter number $m.$ Therefore, $m$ is an eigenvalue of the Laplacian of the commuting graph $\mathcal C_{D_{2m}}$ with multiplicity at least $k(m - |Z(G)|-1).$ Moreover, if $m>2$ is odd, $|Z(G)|=1$ and then the multiplicity of $m$ will be at least $k(m-2).$
					\item  \label{sym}
					Consider the permutation group $\mathfrak S_n$ on $n$ symbols. This is a finite reflection group with $\binom{n}{2}$ reflections. Let $n \geq 3.$  The Coxeter number of $\mathfrak S_n$ is $\frac{2\binom{n}{2}}{n}=n-1.$ So, $n-1$ is an eigenvalue of the Laplacian matrix of the commuting graph $\mathcal C_{\mathfrak S_n}.$ The multiplicity of $n-1$ is at least $k(n-3).$ %We identify $k$ for $n=3,4$ below.
				\end{enumerate}
				
			\end{ex}
			%By \Cref{l-contraction}, $\sigma([L_\G/\mf{G}])\subseteq\sigma(L_{\com{G}} )$. By the multiplicity of an eigenvalue $\lambdsa$ in the spectrum of $\sigma(L_{\com{G}} )\setminus  \sigma([L_\G\setminus \mf{G}])$ we means the multiplicity of $\lambda $ in the set $\sigma(L_{\com{G}} )\setminus  \sigma([L_\G\setminus \mf{G}])$ only and excluding the multiplicity of $\lambda$ because of 
			In the next result, we show that some information of $G$ is encoded in $ \sigma(L_{\com{G}})$.
			
			\begin{cor}\label{g_info_l}
				Let $G$ be a group. 
				\begin{enumerate}
					\item  The cardinality of the center is 
					$$|Z(G)|=(m_1-m_2)+1,$$
					where the multiplicity of $|G|$ in $ \sigma(L_{\com{G}}) $ is $m_1(\ge 0)$, and $m_2(\ge 0)$ is the multiplicity of $|G|$ in $\sigma([L_\G/\mf{G}])$.
					\item For each eigenvalue $\lambda \in  \sigma(L_{\com{G}})\setminus \sigma([L_\G/ \mf{G}])$, there exists $u\in G$, such that $\lambda=C(u)$.
				\end{enumerate}
			\end{cor}
			\begin{proof}
				Since by \Cref{l-contraction}, multiplicity of $ |G|$ in $\sigma(L_{\com{G}})\setminus \sigma([L_\G/ \mf{G}])$ is exactly $|Z(G)|-1$, if the multiplicity of  $ |G|\in \sigma([L_\G/ \mf{G}])$ is $m_2$ (which can be $0$) then $m_1-m_2=|Z(G)|-1$. Thus, the result (1) follows.
				The result (2) directly follows from \Cref{l-contraction}.
			\end{proof}
			In the next result, we show that \Cref{complete-l} leads us to a class of graphs that can not be the commuting graph of any group.
			\begin{cor}\label{nongl}
				Let $\G$ be a graph. If $|V(\G)|$ is not an eigenvalue of $L_{\G}$, then there exists no group $G$ such that $\com{G}=\G$.
			\end{cor}
			\begin{proof}
				If possible, let there exists a group such that $\com{G}=\G$, then by \Cref{complete-l}, $|V(\G)|$ is an eigenvalue of  $L_{\G}$, a contradiction. Thus, our assumption is wrong, and there exists no group $G$ such that  $\com{G}=\G$.
			\end{proof}
			In the next result, we provide a spectral condition to identify the situation when two particular elements of a group can not share the same centralizer.
			\begin{cor}\label{equalcentralizer}
				Let $G$ be a group. For two distinct $u,v\in G$, if there does not exists $\lambda\in \sigma(L_{\com{G}})$ such that $L_{\com{G}}(\chi_{\{u\}}-\chi_{\{v
					\}})=\lambda (\chi_{\{u\}}-\chi_{\{v
					\}})$ then $C(u)\ne C(v)$.
			\end{cor}
			\begin{proof}
				If possible let $C(u)=C(v)$, then by the proof of \Cref{nbd-eig}, and \Cref{complete-l}, $|C(u)|$ is an eigenvalue of $L_{\com{G}}$ with eigenvector $\chi_{\{u\}}-\chi_{\{v
					\}} $, which is a contradiction to the fact that there does not exists $\lambda\in \sigma(L_{\com{G}})$ such that $L_{\com{G}}(\chi_{\{u\}}-\chi_{\{v
					\}})=\lambda (\chi_{\{u\}}-\chi_{\{v
					\}})$ then $C(u)\ne C(v)$. This completes the proof.
			\end{proof}
			So far we have seen $\sigma( [L_\G/\mf{G}])\subseteq  \sigma(L_{\com{G}})$. If $G$ is such that $\mf{G}$ contains a small number of elements, then the order of the matrix $[L_\G/\mf{G}]$ is very small. In that case, one can calculate the spectrum of $[L_\G/\mf{G}]$ easily. 
			
			If $G$ is such that $\mf{G}$ contains many elements, then it might be difficult to calculate $\sigma([L_\G/\mf{G}])$. Therefore, examining certain structures within $G$ that result in symmetry within the graph $\G_{G/\mf{G}}$ is worthwhile.  The symmetries inherent in the structure of $\G_{G/\mf{G}}$ provide some insight into certain eigenvalues of $[L_\G/\mf{G}]$.
			
			A \emph{pendant vertex} in a graph $\G$ is a vertex $v(\in V(\G))$ incident to only one edge in $\G$. For any group $G$, we refer to an element $u\in G\setminus Z(G)$ as a \emph{pendant element} if for all $v\in G\setminus Z(G) $,
			\begin{equation}\label{cond}
				\text{either~}   C(u)=C(v) \text{~or~} C(u)\cap C(v)=Z(G).
			\end{equation}
			we refer two pendant $u,v\in G$ are \emph{equivalent pendant} if $C(u)=C(v)$, otherwise we refer to them as \emph{non-equivalent} pendants. 
			\begin{prop}\label{pendant class}
				Let $G$ be a group. If $u\in G\setminus Z(G)$ is a pendant element then $\mf{F}_u$ is a pendant vertex in $\G_{G/\mf{G}}$.
				%A pendant element $u\in G$ corresponds to a pendant vertex $\mf{F}_u$ in $\G_{G/\mf{G}}$.
			\end{prop}
			\begin{proof}
				If $\mf{F}_0=Z(G)$ then $\mf{F}_u$ is adjacent to $\mf{F}_0$. Now, it is enough to prove that for all $\mf{F}_v\in \mf{G}\setminus \{\mf{F}_0,\mf{F}_u\}$, $ \mf{F}_u$ is not adjacent to $\mf{F}_v$ in $\G_{G/\mf{G}}$. If possible, let, $ \mf{F}_u$ is adjacent to $\mf{F}_v$ in $\G_{G/\mf{G}}$. Thus, $u,v$ commutes in $G$. Therefore, $v\in C(u)\cap C(v)$, and since $ \mf{F}_v\ne \mf{F}_0$, $v\notin Z(G)$. This leads us to 
				$C(u)\cap C(v)\ne Z(G)$. Now, by using \Cref{cond},
				$C(u)=C(v)$, which is a contradiction to the fact $ \mf{F}_u\ne \mf{F}_v$. Therefore, our assumption is wrong, and $\mf{F}_u$ is not adjacent to $\mf{F}_v$ in $\G_{G/\mf{G}}$.
			\end{proof}
			% In the next result, we show that a pair of pendants in a group $G$ leads us to symmetry in $\G_{G/\mf{G}}$, and that symmetry leaves its traces in $\sigma(L_{\com{G}})$.
			The following result demonstrates that the existence of a pair of pendants within a group $G$ results in symmetry within $\G_{G/\mf{G}}$ and that symmetry leaves its traces in $\sigma(L_{\com{G}})$.
			\begin{thm}\label{pendant_l}
				If $G$ contains at least two non-equivalent pendent elements then $ |Z(G)|\in \sigma(L_{\com{G}})$. If the number of the  pendant vertices in $\G_{G/\mf{G}}$ in $G$ is $p$, then the multiplicity of $|Z(G)|$ is at least $p-1$.
			\end{thm}
			\begin{proof}
				Let $G$ be a graph, and $u,v\in G$ are two pendants in $G$. Thus, $\mf{F}_u,\mf{F}_v$ are two pendant vertex in $\G_{G/\mf{G}}$. Our claim is $[L_\G/\mf{G}]x_{uv}=|Z(G)| x_{uv}$, where $x_{uv}:\mf{G}\to\mathbb{R}$, is defined by 
				$$x_{uv}(\mf{F}_i)=
				\begin{cases}
					\phantom{-}|\mf{F}_v| &\text{~if~} \mf{F}_i=\mf{F}_u,\\
					- |\mf{F}_u| &\text{~if~} \mf{F}_i=\mf{F}_v,\\
					\phantom{-}0&\text{~otherwise.}
				\end{cases}$$
				Let $\mf{G}=\{\mf{F}_0,\ldots,\mf{F}_m\}$, with $\mf{F}_0=Z(G)$. Now, we get the following.
				\begin{enumerate}
					\item For $\mf{F}_i\in \mf{G}\setminus \{\mf{F}_0,\mf{F}_u,\mf{F}_v\}$, since $\mf{F}_i$ is not adjacent to $\mf{F}_u$, $\mf{F}_v$, either $l_{ij}=0$ or $x_{uv}(\mf{F}_j)=0$, for all $ \mf{F}_j\in\mf{G}$. Therefore,
					\begin{align*}
						( [L_\G/\mf{G}] x_{uv})(\mf{F}_w)&=\d\limits_{j:\mf{F}_j\in\mf{G}}l_{ij}x_{uv}(\mf{F}_j)=0.
					\end{align*}
					\item  For $\mf{F}_i=\mf{F}_u$, since $\mf{F}_i$ is adjacent to only $\mf{F}_0$,
					\begin{align*}
						( [L_\G/\mf{G}] x_{uv})(\mf{F}_i)&=\d\limits_{j:\mf{F}_j\in\mf{G}}l_{ij}x_{uv}(\mf{F}_j)\\
						&=l_{ii}x_{uv}(\mf{F}_i)+l_{i0}x_{uv}(\mf{F}_0)\\
						&=|\mf{F}_0||\mf{F}_v|=|Z(G)|x_{uv})(\mf{F}_i)\\
						&=|Z(G)|x_{uv}(\mf{F}_i).
					\end{align*}
					\item  For $\mf{F}_i=\mf{F}_v$, since $\mf{F}_i$ is adjacent to only $\mf{F}_0$, similar to the previous case, we have
					$
					( [L_\G/\mf{G}] x_{uv})(\mf{F}_i)=|Z(G)|x_{uv}(\mf{F}_i).
					$
					\item For $\mf{F}_i=\mf{F}_0$, 
					\begin{align*}
						( [L_\G/\mf{G}] x_{uv})(\mf{F}_0)&=\d\limits_{j:\mf{F}_j\in\mf{G}}l_{ij}x_{uv}(\mf{F}_j)\\
						&=|\mf{F}_u||\mf{F}_v|-|\mf{F}_u||\mf{F}_v|=0.
					\end{align*}
				\end{enumerate}
				Therefore, $[L_\G/\mf{G}] x_{uv}=|Z(G)|x_{uv}$. Thus, by \cref{l-contraction}, $|Z(G)|$ is an eigenvalue of $L_{\com{G}}$ with eigenvector $\ov{x_{uv}}:G\to \mathbb{R}$ defined by $\ov{x_{uv}}(w)=x_{uv}(\mf{F}_w)$ for all $w\in G$. If $\{v_1,\ldots, v_p\}$ is the collection of all the pendants in $G$, then $\ov{x_{p1}},\ldots,\ov{x_{pp-1}}$ is the collection of $p-1$ independent eigenvectors corresponding the eigenvalue $|Z(G)|$ of $L_{\com{G}}$. Thus, the multiplicity of the eigenvalue $|Z(G)|$ is at least $p-1$.
			\end{proof}
			\begin{ex}
				In \Cref{fig:contraction}, we can see $\Gamma_{\mf{S}_4/\mathfrak{G}}$ has 
				$4$ pendants, namely $\mf{F}_9$, $\mf{F}_4$, $\mf{F}_{12}$, and $\mf{F}_{16}$. Thus, $|Z(\com{\mf{S}_4})|=1
				$ is an eigenvalue of $L_{\com{\mf{S}_4}}$ with multiplicity $3$. 
			\end{ex}
			Using \Cref{g_info_l}, we can find the value of $|Z(G)|$ encrypted in $\sigma(L_{\com{G}})$, and along with this information, \Cref{pendant_l} leads us to the following result that enables us to decode another information of a group $G$ from $\sigma(L_{\com{G}})$.
			\begin{cor}\label{zg-cor}
				Let $G$ be a group if $|Z(G)|\notin \sigma(L_{\com{G}})$, then $\G_{G/\mf{G}}$ has at most one  pendant.
			\end{cor}
			% The converse of \Cref{zg-cor} is not true. That is, $|Z(G)|\in \sigma(L_{\com{G}})$ can not ensure the existence of more than one pendant in $G$. For justification, consider the following example.

			It might be a lengthy process to calculate the matrix $[L_\G/\mf{G}]$ and its spectrum. Now we provide a favourable situation, where $\sigma([L_\G/\mf{G}])$ can be calculated easily. Since the next result directly follows from \Cref{complete-l}, and \Cref{pendant_l}, we state it without proof.
			\begin{cor}\label{all-pendant}
				Let $G$ be a group.
				If for all $u,v\in G\setminus Z(G)$  
				\begin{equation}\label{con} \text{~either~} C(u)=C(v) \text{~or~} C(u)\cap C(v)=Z(G),\end{equation} 
				then the complete spectrum of $L_{\com{G}}$ is the following.
				\begin{enumerate}
					\item  $|C(u)|$ is an eigenvalue of multiplicity $|\mf{F}_u|-1$  for each $u\in G\setminus Z(G)$.
					\item $|G|$ is an eigenvalue of multiplicity $|Z(G)| $.  
					\item $|Z(G)|$ is an eigenvalue with multiplicity $|\mf{G}|-2$.
					\item $0$ is an eigenvalue with multiplicity $1$.
				\end{enumerate}
			\end{cor}
			\begin{proof}
				Since the condition given by \Cref{con} holds, by \Cref{pendant class}, all the $\mf{F}_i(\ne Z(G))\in \mf{G}$ are pendant vertices in  $\G_{G/\mf{G}}$. Thus, the result follows from \Cref{pendant_l}.
			\end{proof}
			If all the $\mf{F}_i(\ne Z(G))\in \mf{G}$ are pendants, then $\ov{\G_{G/\mf{G}}}$ has two connected components, and thus, multiplicity of $|G|$ is $|Z(G)|$. Similarly, for any group $G$, if $\ov{\G_{G/\mf{G}}}$ has exactly two connected components then  multiplicity of the eigenvalue $|G|$ of $L_{\com{G}}$ is $|Z(G)|$.
			
			Now we illustrate our results using some examples.
			A group is said to be \emph{centralizer abelian group} if the centralizer of any nonidentity element is an abelian subgroup, see \cite[p. 686]{MR86818}, \cite[p. 291]{MR815926}. The following lemma shows that \Cref{all-pendant} describes the complete list of eigenvalues (with multiplicities) of the Laplacian matrix of the graph $\m C_G$ while $G$ is a centralizer abelian group.
			\begin{lem}
				Let $G$ be a centralizer abelian group. Then all the elements of $G$ are pendants.
			\end{lem}
			\begin{proof}
				Suppose $u,v \in G$ such that $Z(G) \subsetneq C(u)\cap C(v).$ Then there exists $a \notin Z(G)$ such that $a \in C(u)\cap C(v).$ Note that $a \in C(u),$ so $u \in C(a).$ Again, $C(a)$ is abelian, so if $x \in C(a),$ then $xu=ux$, that is, $x \in C(u).$ In other words, $C(a) \subseteq C(u).$ Similarly, $C(u) \subseteq C(a).$ Therefore, $C(u)= C(a)$. Also, similar arguments as above prove that $C(a) = C(v).$ Hence, the result follows.
			\end{proof}
			
			An abelian group is always a centralizer abelian group. Among the following examples, $\mathfrak S_3$ is a centralizer abelian group. However, the groups $\mathfrak S_4,$ $D_8$ and $Q_8$  are not centralizer abelian groups.
			
			\begin{ex}[$\mathfrak S_3$]\rm \label{s3} 
				The symmetric group $\mathfrak S_3=\{v_0,v_1,\ldots, v_5\}$ is a non-abelian group, where $v_0$ is the identity permutation, $v_1=(1,2,3)$, $v_2=(1,3,2)$, $v_3=(1,2)$, $v_4=(2,3)$, $ v_5=(1,3)$.
				It is easy to see that $Z(\mathfrak S_3)=\{v_0\}$. Note that \begin{itemize}
					\item $C(v_1)=C(v_2)=\{v_0,v_1,v_2\}$ and
					\item $C(v_3)=\{v_0,v_3\}, ~ C(v_4)=\{v_0,v_4\}, ~ C(v_5)=\{v_0,v_5\}.$
				\end{itemize} 
				In the symmetric group on $3$ symbols, all the elements are pendants.
				%satisfies the condition given in Equation \eqref{con}.
				
				\begin{figure}[ht]
					\centering
					\begin{subfigure}[b]{0.49\textwidth}
						\centering
						\begin{tikzpicture}
							\draw[](-2.02,1.09)circle (2.5pt);
							\draw[](-2.3,1.35) node{$v_1$};
							\draw[](-2.02,-1.09)circle (2.5pt);
							\draw[](-2.3,-1.35) node{$v_2$};
							\filldraw[color=black](-1,0)circle (2.5pt);
							\draw[](-1,0.5) node{$v_0$};
							\draw[](0.07,0)circle (2.5pt);
							\draw[](0.4,0) node{$v_4$};
							\draw[](0.07,1.03)circle (2.5pt);
							\draw[](0.4,1.35) node{$v_3$};
							\draw[](0.07,-1.03)circle (2.5pt);
							\draw[](0.4,-1.35) node{$v_5$};
							\draw[] (-2,1)--(-2,-1)--(-1,0)--(-2,1);
							\draw[](-1,0)--(0,1);
							\draw[](-1,0)--(0,0);
							\draw[](-1,0)--(0,-1);
						\end{tikzpicture}
						\label{fig:s3}
						\caption{$\mathcal{C}_{\mathfrak S_3}$: Commuting Graph associated to $\mathfrak S_3$}
					\end{subfigure}
					\hfill
					\begin{subfigure}[b]{0.49\textwidth}
						\centering
						\begin{tikzpicture}
							\draw[](-2.3,1.35) node{$\begin{pmatrix}
									5&-1&-1&-1&-1&-1\\-1&2&-1&0&0&0\\-1&-1&2&0&0&0\\-1&0&0&1&0&0\\-1&0&0&0&1&0\\-1&0&0&0&0&1
								\end{pmatrix}$};
						\end{tikzpicture}
						\caption{L : Laplacian matrix of $\mathcal{C}_{\mathfrak S_3}$}
						\label{fig:laps3}
					\end{subfigure}
					\hfill
				\end{figure}
				We write the complete list of eigenvalues of $L_{\mathcal{C}_{\mathfrak S_3}}$ here using \Cref{all-pendant}. We arrange the eigenvalues in increasing order by $0 = \lambda_1 \leq \ldots \leq \lambda_6.$
				\begin{enumerate}[leftmargin=*]
					\item By \Cref{all-pendant} (1), $|C(v_1)|= 3$ with multiplicity  $|C(v_1)\setminus Z(\mathfrak S_3)|-1= 1$.
					\item It follows from \Cref{all-pendant} (2) that the eigenvalue $|\mathfrak S_3|= 6$ is with multiplicity $|Z(G)|=1.$
					
					\item 
					% Note that the number of disjoint components of the graph $\Gamma(\mathfrak S_3\setminus Z(\mathfrak S_3), E)$ is $r=4$. Therefore,
					By \Cref{all-pendant} (3), we get that $|Z(\mathfrak S_3)|=1$ is an eigenvalue with multiplicity $|\mf{G}|-2=3$. That implies $\lambda_2 = \lambda_3 =\lambda_4 = 1.$
					
					\item Since $\mathcal{C}_{\mathfrak S_3}$ is a connected graph, $\lambda_1 =0$ with multiplicity $1$.
				\end{enumerate}
				It is intriguing to note that $|C(v_i)|=2$ for $i=3,4,5,$ are not included in the list of eigenvalues as 
				%the multiplicity of $2$ is
				$|\mf{F}_i|-1=0$, for all $i=3,4,5$. This can be seen from \Cref{pseu}(\ref{sym}) as well. 
			\end{ex}
			\begin{figure}[ht]
				\centering
				\begin{subfigure}[b]{0.55\textwidth}
					\centering
					\rotatebox{270}{  \begin{tikzpicture}[scale=0.25]
							\filldraw[color=black](0,2)circle (10pt); 
							\draw[color=black,very thick](0,3.5) node[rotate=90]{$id$};
							\draw[color=black,very thick](0,-3.5) node[rotate=90]{$x^2$};
							\filldraw[color=black](0,-2)circle (10pt); 
							\draw[](-5,2)circle (10pt);
							\draw(-5.3,3.5) node[rotate=90]{$x$};
							\draw[](-5,-2)circle (10pt); 
							\draw (-5.8,-3.5) node[rotate=90]{$x^3$};
							\draw[](5,2)circle (10pt);
							\draw (6.8,2) node[rotate=90]{$x^2y$};
							\draw[](5,-2)circle (10pt); 
							\draw[](7,-2) node[rotate=90]{$xy$};
							\draw[](5,6)circle (10pt);
							\draw[](7,6.5) node[rotate=90]{$y$};
							\draw[](5,-6)circle (10pt); 
							\draw[](6.8,-6.5) node[rotate=90]{$x^3y$};
							
							\draw[](0.1,1.6)--(0,-1.7);
							\draw[](-4.9,1.7)--(-5,-1.7);
							\draw[](-0.2,1.8)--(-4.8,-1.7);
							\draw[](-4.7,1.9)--(-0.2,-1.8);
							\draw[](0.3,2)--(4.7,2);
							\draw[](-4.7,2)--(-0.3,2);
							\draw[](-4.7,-2)--(-0.3,-2);

							\draw[](0.3,-2)--(5,1.7);
							\draw[](0.3,2.3)--(4.7,6);
							\draw[](0.1,-1.7)--(4.9,5.7);
							\draw[](5,2.3)--(5,5.7);

							\draw[](0.3,-2)--(4.7,-2);
							\draw[](0.3,-2.3)--(4.7,-6);
							\draw[](5,-2.3)--(5,-5.7);
							\draw[](4.8,-1.8)--(0.2,1.8);
							\draw[](0.1,1.7)--(4.9,-5.8);
					\end{tikzpicture}  } 
					\caption{$\mathcal{C}_{D_8}$}
					\label{fig:D}
				\end{subfigure}
				\hfill
				\begin{subfigure}[b]{0.44\textwidth}
					\centering
					\rotatebox{270}{ \begin{tikzpicture}[scale=0.25]
							\filldraw[color=black](0,2)circle (10pt);  \draw[color=black](0,4) node[rotate=90]{$id$};
							\draw[color=black,very thick](0,-4) node[rotate=90]{$a$};
							\filldraw[color=black](0,-2)circle (10pt); 
							\draw[](-5,2)circle (10pt);
							\draw[](-5,4) node[rotate=90]{$b$};
							\draw[](-5,-2)circle (10pt); 
							\draw[](-5,-4) node[rotate=90]{$ab$};
							\draw[](5,2)circle (10pt);
							\draw[](7,2) node[rotate=90]{$ac$};
							\draw[](5,-2)circle (10pt); 
							\draw[](7,-2) node[rotate=90]{$ad$};
							\draw[](5,6)circle (10pt);
							\draw[](7,6) node[rotate=90]{$c$};
							\draw[](5,-6)circle (10pt); 
							\draw[](7,-6) node[rotate=90]{$d$};
							
							\draw[](0.1,1.6)--(0,-1.7);
							\draw[](-4.9,1.7)--(-5,-1.7);
							\draw[](-0.2,1.8)--(-4.8,-1.7);
							\draw[](-4.7,1.9)--(-0.2,-1.8);
							\draw[](0.3,2)--(4.7,2);
							\draw[](-4.7,2)--(-0.3,2);
							\draw[](-4.7,-2)--(-0.3,-2);

							\draw[](0.3,-2)--(5,1.7);
							\draw[](0.3,2.3)--(4.7,6);
							\draw[](0.1,-1.7)--(4.9,5.7);
							\draw[](5,2.3)--(5,5.7);

							\draw[](0.3,-2)--(4.7,-2);
							\draw[](0.3,-2.3)--(4.7,-6);
							\draw[](5,-2.3)--(5,-5.7);
							\draw[](4.8,-1.8)--(0.2,1.8);
							\draw[](0.1,1.7)--(4.9,-5.8);
					\end{tikzpicture}   }
					\caption{$\mathcal{C}_{Q_8}$}
					\label{fig:Q}
				\end{subfigure}
				\hfill
				\caption{Commuting graphs associated to $D_8$ and $Q_8$}
				\label{fig:DQ}
			\end{figure}
			
			\begin{ex}[Groups of order $8$]\label{oder8}\rm
				Let $\mb{Z}_n$ denote the cyclic group of order $n$. Up to isomorphism, there are exactly $5$ groups of order $8$ and those are $\mb{Z}_8$,
				$\mathbb{Z}_4\times \mathbb{Z}_2 $,
				$\mathbb{Z}_2\times \mathbb{Z}_2\times \mathbb{Z}_2$,
				$ D_8$ (Dihedral Group),
				$ Q_8$ (Quaternion Group). We enlist some observations on the commuting graph associated with any group of order $8$.
				
				\textbf{Abelian groups:}
				Note that the first three groups of the above list are Abelian. Hence, the commuting graph associated with each of those is given by the complete graph $K_8$ and thus, the eigenvalues of the Laplacian of the commuting graph are $0$ and $8$ with multiplicity $1$ and $7$, respectively.
				
				The groups $D_8$ and $Q_8$ are non-abelian. We provide a detailed description of the Laplacian spectrum of $\mathcal{C}_{D_8}$ and $\mathcal{C}_{Q_8}$ below.
				
				\begin{enumerate}[leftmargin=*]
					\item \textbf{Dihedral group}\label{D}
					The set of symmetries of a square is given by the Dihedral group $D_8= \langle x,y: x^4=y^2=id \text{~and~} yx=x^3y\rangle.$ %\{i,x,x^2,x^3,y,xy,x^2y,x^3y\} 
					Note that 
					\begin{itemize}
						\item $Z(D_8)=\{id,x^2\}$,
						\item $C(x^2y)=C(y)=\{id,x^2,y,x^2y\}$,
						\item $C(xy)=C(x^3y)=\{id,x^2,xy,x^3y\}$,
						\item $C(x)=C(x^3)=\{id,x,x^2,x^3\}$.
					\end{itemize}
					Clearly, $D_8$ is a non-abelian group that satisfies the condition in Equation \eqref{con}. The eigenvalues of the Laplacian matrix $L_{\mathcal C_G}$ arranged in increasing order by $0= \lambda_1 \leq \ldots \leq \lambda_8,$ are described below.
					\begin{enumerate}
						\item Since $L_{\com{D_8}}$ is Laplacian of a connected graph, $\lambda_1 =0$ with multiplicity $1$.
						\item By \Cref{all-pendant}(2), we get that $|Z(D_8)|=2$ is an eigenvalue with multiplicity $2$. That implies $\lambda_2 = \lambda_3 = 2.$
						\item Again, by \Cref{all-pendant}(1), $\lambda_4=|C(y)|= 4$ with multiplicity  $|\mf{F}_y|-1= 1$,
						\item $\lambda_5=|C(xy)|= 4$ with multiplicity  $|\mf{F}_{xy}|-1=1$, and
						\item $\lambda_6=|C(x)|= 4$ with multiplicity  $|\mf{F}_{x}|-1= 1.$
						\item Moreover, it follows from \Cref{all-pendant} (2) that the eigenvalue $|D_8|= 6$ is with multiplicity $|Z(G)|=2.$ Hence $\lambda_7=\lambda_8=8$.
					\end{enumerate}
					It is interesting to note that $D_8$ is a finite pseudoreflection group with Coxeter number $4.$ By Example \ref{dihed}, the eigenvalue $4$ will repeat at least $k(4-2-1)=k$ times. Here, $k=2$ is less than the multiplicity of the eigenvalue $4.$
					\item \textbf{(Quaternion group) }\label{Q}\rm
					%The  quaternion number system \cite{hamilton1844quaternions} is a generalization of the complex number system. 
					Consider the Quaternion group $ Q_8=\langle a,b,c,d:a^2=id,b^2=c^2=d^2=bcd=a\rangle$, where $id$ is the identity element. %The group $Q_8$ is isomorphic to an eight-element subset of the quaternions. Hence $Q_8$ is called Quaternion group. 
					Note that
					\begin{itemize}
						\item $Z(Q_8)=\{id,a\}$,
						\item $C(b)=C(ab)=\{id,a,b,ab\}=\{id,b^2,b,b^3\}$,
						\item $C(c)=C(ac)=\{id,a,c,ac\}=\{id,c^2,c,c^3\}$,
						\item $C(d)=C(ad)=\{id,a,d,ad\}=\{id,d^2,d,d^3\}$.
					\end{itemize}
					Clearly, $Q_8$ is a non-abelian group which satisfies the condition given in Equation \eqref{con}. The associated commuting graph $\mathcal C_{Q_8}$ is drawn in Figure \ref{fig:Q}. Clearly, the graph $C_{Q_8}$ and $\mathcal C_{D_8}$ are identical.  Hence, the complete list of eigenvalues of the Laplacian matrix of $\mathcal{C}_{Q_8}$ is same as the Laplacian spectrum of the graph $\mathcal C_{D_8}.$
					
					Though $D_8$ and $Q_8$ are not isomorphic to each other, the associated commuting graphs $\mathcal{C}_{D_8}$ and $\mathcal{C}_{Q_8}$, drawn in Figure \ref{fig:D} and  Figure \ref{fig:Q} respectively, are same. So is the Laplacian matrices associated to $\mathcal{C}_{D_8}$ and $\mathcal{C}_{Q_8}$. We denote it by $L_8.$ Therefore, one has $$ L_8 =  \begin{pmatrix}
						7&-1&-1&-1&-1&-1&-1&-1\\-1&7&-1&-1&-1&-1&-1&-1\\-1&-1&3&-1&0&0&0&0\\-1&-1&-1&3&0&0&0&0\\-1&-1&0&0&3&-1&0&0\\-1&-1&0&0&-1&3&0&0\\-1&-1&0&0&0&0&3&-1\\-1&-1&0&0&0&0&-1&3
					\end{pmatrix}.$$
					One can easily check that if two finite groups $G_1$ and $G_2$ are isomorphic to each other, then $\m C_{G_1} = \Gamma(G_1, E_1)$ and $\m C_{G_2} = \Gamma(G_2, E_2)$ are identical, that is, there exists a bijection between $E_1$ and $E_2.$ But the above examples show that the converse is not necessarily true.

				\end{enumerate}          
			\end{ex}
			A graph $\G$ is called \emph{Laplacian integral} if $\sigma(L_{\G})\subseteq\mathbb{Z} $, the set of all integers. Since all the eigenvalues provided in \Cref{pendant_l} are integers, we have the following result.
			\begin{cor}
				Let $G$ be a group.
				If  for all $u,v\in G\setminus Z(G)$  
				$ \text{~either~} C(u)=C(v) \text{~or~} C(u)\cap C(v)=Z(G),$
				then $\com{G}$ is Laplacian integral. 
			\end{cor}
			In \Cref{s3}, and \Cref{oder8}, we have seen $\mf{S}_3$, $D_8$, and $Q_8$ are Laplacian integral.
			% In \cref{exmp}, we provide some examples of  group $G$, where   If all the $\mf{F}_i(\ne Z(G))\in \mf{G}$ are pendants, and thus, we can apply \Cref{all-pendant} to provide the complete spectra of $L_{\com{G}}$.

			The condition in \Cref{con} implies that $u,v\in G\setminus Z(G)$ are such that each of $\mf{F}_u$, and $\mf{F}_v$ has only one neighbour $\mf{F}_0=Z(G)$ in $\G_{G/\mf{G}}$. Now we will consider the situation where $\mf{F}_u$, and $\mf{F}_v$ can have more than one neighbour but they are non-adjacent twins in $\G_{G/\mf{G}}$.
			\begin{thm}\label{non_nbd}
				Let $G$ be a group. If $u,v\in G$ such that
				\begin{enumerate}
					\item $uv\ne vu$,
					\item   $uw=wu$ if and only if $vw=wu$ for all $w\in G\setminus(\mf{F}_u\cup \mf{F}_v)$,
				\end{enumerate}
				then $|C(u)\cap C(v)|$ is an eigenvalue of $L_{\com{G}}$.
			\end{thm}
			\begin{proof} To prove this result, we show that $|C(u)\cap C(v)|$ is an eigenvalue of $[L_\G/\mf{G}]$.
				Since  $uv\ne vu$, we have $C(u)\ne C(v)$, and $\mf{F}_u\ne\mf{F}_v$. Since $uw=wu$ if and only if $vw=wu$,  for all $w\in G\setminus(\mf{F}_u\cup \mf{F}_v)$, in $\G_{G/\mf{G}}$ two vertices $\mf{F}_u$, $\mf{F}_v$, have the same neighbours, that is $N_{\G_{G/\mf{G}}}(u)=N_{\G_{G/\mf{G}}}(v)$. Therefore, $$\sum\limits_{\mf{F}_j\in N_{\G_{G/\mf{G}}}(\mf{F}_u)}|\mf{F}_j|=\sum\limits_{\mf{F}_j\in N_{\G_{G/\mf{G}}}(\mf{F}_v)}|\mf{F}_j|=|C(u)\cap C(v)|.$$

				Our claim is $[L_\G/\mf{G}]x_{uv}=|C(u)\cap C(v)| x_{uv}$, where $x_{uv}:\mf{G}\to\mathbb{R}$ is defined by 
				$$x_{uv}(\mf{F}_i)=
				\begin{cases}
					\phantom{-}|\mf{F}_v| &\text{~if~} \mf{F}_i=\mf{F}_u,\\
					- |\mf{F}_u| &\text{~if~} \mf{F}_i=\mf{F}_v,\\
					\phantom{-}0&\text{~otherwise.}
				\end{cases}$$
				Let $\mf{G}=\{\mf{F}_0,\ldots,\mf{F}_m\}$, with $\mf{F}_0=Z(G)$. Now, we get the following.
				\begin{enumerate}
					\item For $\mf{F}_i=\mf{F}_w\in\mf{G}\setminus\{\mf{F}_u,\mf{F}_v\}$, where $\mf{F}_w$ is adjacent to both $\mf{F}_u$, and $\mf{F}_u$,
					\begin{align*}
						(  [L_\G/\mf{G}]x_{uv})(\mf{F}_i)&=|\mf{F}_u|x_{uv}(\mf{F}_u)+|\mf{F}_v|x_{uv}(\mf{F}_v)=0.
					\end{align*}
					\item If $\mf{F}_i$ is not adjacent with $\mf{F}_u$, and $\mf{F}_v$ then either $l_{ij}=0$ or $x_{uv}(\mf{F}_j)=0$. Thus, 
					\begin{align*}
						(  [L_\G/\mf{G}]x_{uv})(\mf{F}_i)=0.
					\end{align*}
					\item If $\mf{F}_i=\mf{F}_u$, then $l_{ii}=\sum\limits_{j(\ne i):j\sim i}|\mf{F}_j|=\sum\limits_{\mf{F}_j\in N_{\G_{G/\mf{G}}}(\mf{F}_u)}|\mf{F}_j|=|C(u)\cap C(v)|$. Thus,
					\begin{align*}
						(  [L_\G/\mf{G}]x_{uv})(\mf{F}_i)&=l_{ii}x_{uv}(\mf{F}_u)=|C(u)\cap C(v)|x_{uv}(\mf{F}_u).
					\end{align*}
					\item Similarly, for $\mf{F}_i=\mf{F}_v $, we can show $(  [L_\G/\mf{G}]x_{uv})(\mf{F}_i)=|C(u)\cap C(v)|x_{uv}(\mf{F}_v)$.
				\end{enumerate}
				Therefore, $ [L_\G/\mf{G}]x_{uv}=|C(u)\cap C(v)|x_{uv}$. Thus, the result follows.
			\end{proof}
			For $u,v\in G\setminus Z(G)$ with $C(u)\cap C(v)=Z(G)$, we have $|Z(G)|=|C(u)\cap C(v)|$ is an eigenvalue of $L_{\com{G}}$ that is \Cref{pendant_l} can be proved as a corollary of \Cref{non_nbd}.
			\begin{ex}
				Since $18,2\in \mf{S}_4$ (see \Cref{fig:S-4}, and \Cref{s4initial}) does not commutes, but they share the same neighbours outside of $\mf{F}_2\cup\mf{F}_{18}$, by \Cref{non_nbd}, $ |C(2)\cap C(18)|=2$ is an eigenvalue of $L_{\com{G}}$.  Since $|C(2)\cap C(18)|=|C(3)\cap C(11)|=|C(10)\cap C(6)|=2 $, by the same argument, we have the multiplicity of $2$ is at least $3$.
			\end{ex}
			\subsection{Signless Laplacian spectra of the commuting graph of a group }
			Consider the matrix $[Q_{\com{G}}/\mf{G}]=(q_{ij})_{\mf{F}_i,\mf{F}_j\in\mf{G }}$ associated with $\G_{G/\mf{G}}$, defined as
			$$q_{ij}=\begin{cases}
				|\mf{F}_j|&\text{~if~}i\ne j \text{~and~} \mf{F}_i \text{ is adjacent to }\mf{F}_j,\\
				0&\text{~if~}i\ne j \text{~and~} \mf{F}_i \text{ is not adjacent to }\mf{F}_j,\\
				2(|\mf{F}_i|-1)+\sum\limits_{j(\ne i)=0}^ml_{ij} &\text{~if~}i=j.
			\end{cases} $$
			In the next result, we show any eigenvalue of $[Q_{\com{G}}/\mf{G}]$ is also an eigenvalue of $ Q_{\com{G}}$.
			\begin{thm}\label{q-contraction}
				Let $G$ be a group. If $\lambda$ is an eigenvalue of the matrix $[Q_{\com{G}}/\mf{G}]$ with eigenvector $x:\mf{G}\to\mathbb{R}$, then
				% $=(q_{ij})_{\mf{F}_i,\mf{F}_j\in\mf{G }}$ of order $|\mf{G}|$, defined by
				% $$q_{ij}=\begin{cases}
					%   |\mf{F}_j|&\text{~if~}i\ne j \text{~and~} \mf{F}_i \text{ is adjacent to }\mf{F}_j,\\
					%    0&\text{~if~}i\ne j \text{~and~} \mf{F}_i \text{ is not adjacent to }\mf{F}_j,\\
					%    2(|\mf{F}_i|-1)+\sum\limits_{j(\ne i)=0}^ml_{ij} &\text{~if~}i=j,
					% \end{cases} $$
				$\lambda$ is also an eigenvalue of $Q_{\com{G}}$, with eigenvector $\ov{x}:{G}\to\mathbb{R}$, defined by $\ov{x}(v)=x(\mf{F}_v)$ for all $v\in V(\G)$.
			\end{thm}
			\begin{proof}
				Let $v\in \mf{F}_i$, for some $i=1,\ldots,m$,
				\begin{align*}
					(Q_{\com{G}}\ov{x})(v)&=\sum\limits_{u(\ne v)\in C(v)}(\ov{x}(v)+\ov{x}(u))\\
					&=\sum\limits_{u(\ne v)\in \mf{F}_i}(\ov{x}(v)+\ov{x}(u))+\sum\limits_{j(\ne i):j\sim i}\sum\limits_{u\in \mf{F}_j}(\ov{x}(v)+\ov{x}(u))\\
					&=(|\mf{F}_i|-1)2x(\mf{F}_i)+\sum\limits_{j(\ne i):j\sim i}|\mf{F}_j|(x(\mf{F}_i)+x(\mf{F}_j))\\
					&=\left[2(|\mf{F}_i|-1)+\sum\limits_{j(\ne i):j\sim i}|\mf{F}_j|\right]x(\mf{F}_i)+\sum\limits_{j(\ne i):j\sim i}|\mf{F}_j|x(\mf{F}_j)\\
					&=([Q_{\com{G}}/\mf{G}]x)(\mf{F}_i).
				\end{align*}
				Thus, the result follows.
			\end{proof}
			The following result provides the complete spectra of $Q_{\com{G}}$.
			\begin{thm}[Complete spectrum of $Q_{\com{G}}$] \label{complete-q}
				Let $G$ be a group. The complete spectrum of $Q_{\com{G}}$ is listed below.
				\begin{enumerate}
					\item For all $v\in G$, $|C(v)|-2$ is an eigenvalue of multiplicity $|\mf{F}_v|-1$.
					\item The remaining eigenvalues are the eigenvalues of $[Q_{\com{G}}/\mf{G}] $, that is
					$\sigma([Q_{\com{G}}/\mf{G}])\subseteq \sigma(Q_{\com{G}})$. If $x\in \Omega_{\lambda}([Q_{\com{G}}/\mf{G}])$ for any $\lambda\in \sigma([Q_{\com{G}}/\mf{G}])$, then the blow up $\ov{x}\in \Omega_{\lambda}(L_{\com{G}})$.
					
					% For any eigenvalue $\lambda$, with eigenvector $x:\mf{G}\to\mathbb{R}$ of the matrix $[Q_{\com{G}}/\mf{G}]=(q_{ij})_{\mf{F}_i,\mf{F}_j\in\mf{G }}$ of order $|\mf{G}|$, defined by
					% $$q_{ij}=\begin{cases}
						% 	|\mf{F}_j|&\text{~if~}i\ne j \text{~and~} \mf{F}_i \text{ is adjacent to }\mf{F}_j,\\
						% 	0&\text{~if~}i\ne j \text{~and~} \mf{F}_i \text{ is not adjacent to }\mf{F}_j,\\
						% 	2(|\mf{F}_i|-1)+\sum\limits_{j(\ne i)=0}^mq_{ij} &\text{~if~}i=j,
						% \end{cases} $$
					% $\lambda$ is also an eigenvalue of $Q_{\com{G}}$, with eigenvector $\ov{x}:{G}\to\mathbb{R}$, defined by $\ov{x}(v)=x(\mf{F}_v)$, for all $v\in V(\G)$.
				\end{enumerate}
			\end{thm}
			\begin{proof}
				The result follows from \Cref{nbd-eig} and \Cref{q-contraction}. We just need to show that the list is complete. %The total number of eigenvalues of $Q_{\com{G}}$ is $|G|$.
				
				Suppose that $ \mf{G}=\{\mf{F}_0,\ldots,\mf{F}_m\}$, and $\mf{F}_i=\{v_{ij};j=0,1,\ldots,k_i\}$.
				Number of eigenvalues of the form $|C(v)|$, for all $v\in V(G)$ is $|G|-|\mf{G}|$. The set of correpnoding $|G|-|\mf{G}|$ linearly independent eigenvectors is $\mf{E}_{\mf{G}}=\{\chi_{u_{ij}}-\chi_{u_{i0}}:j=1,\ldots,k_i,i=0,\ldots,m\}$. The remaining $|\mf{G}|$ eigenvectors are $\ov{y_i}$, for $i=0,1,\ldots,m$, where $ \lambda_i$ is an eigenvalue of $[Q_{\com{G}}/\mf{G}]$ with eigenvector $y_i$, for $i=0,\ldots,m$. Since with respect to the usual inner product, each $y_i$ is orthogonal to any eigenvectors in $\mf{E}_{\mf{G}}$. Thus, $\mf{E}_{\mf{G}}\cup\{\ov{y_i}:i=0,1,\ldots,m\} $ is the complete list of eigenvectors of $Q_{\com{G}}$. Therefore, the list given in this result is the complete spectrum of $Q_{\com{G}}$.
			\end{proof}
			\begin{ex}\rm
				Using \Cref{complete-q}, we have the following eigenvalues of $Q_{\com{\mf{S}_4}}$.
				\begin{enumerate}[leftmargin=*]
					\item Since $|C(v)|=4$, and $|\mf{F}_v|=2$ for $v=7,3,10,11,6,18\in \mf{S}_4$ (see \Cref{fig:S-4}), by \Cref{complete-q}(1), $2$ is an eigenvalue of multiplicity at least $6$.
					\item Similarly, since $ |\mf{F}_v|-1=1$ for $v=4,9,12, 16\in \mf{S}_4$, by \Cref{complete-q}(1), $ |C(v)|-2=1$ is an eigenvalue of multiplicity at least $ 4$.
					\item Other eigenvalues are the eigenvalues of $[Q_{\com{\mf{S}_4}}/\mf{G}]$.
				\end{enumerate}
			\end{ex}
			For a linear operator $M$, the spectral radius $\rho(M)=\max\{|\lambda|:\lambda\in \sigma(M)\}$. For any group $G$, the graph $\com{G}$ is connected. Thus, $Q_{\com{G}}$ is a non-negative, irreducible matrix. By Perron-Frobenius theorem (\cite[Theorem 8.4.4]{MR2978290}), the spectral radius, $\rho(Q_{\com{G}}) \in \sigma(Q_{\com{G}}) $, the multiplicity of $\rho(Q_{\com{G}}) $ is $1$, and there exists an eigenvector $x:G\to\mathbb{R}$ of $\rho(Q_{\com{G}}) $ such that $x(v)>0$ for all $v\in G$. Thus, we have the following result.
			\begin{cor}\label{rhoq}
				For any group $G$, $\rho(Q_{\com{G}})=\rho([Q_{\com{G}}/\mf{G}]) $. 
			\end{cor}
			\begin{proof}
				By \Cref{complete-q}, and the proof of \Cref{nbd-eig}, if $\lambda\in \sigma(Q_{\com{G}})\setminus\sigma([Q_{\com{G}}/\mf{G}])$, then there exists no eigenvector $x$ corresponding to $\lambda$ such that $x(v)>0$ for all $v\in G$. Thus, the result follows from the Perron-Frobenius theorem.
			\end{proof}
			%Since for all $v\in $
			For any group $G$, by \Cref{complete-q}, one can compute the complete spectra of $Q_{\com{G}}$ just by computing the spectra of a smaller matrix $[Q_{\com{G}}/\mf{G}]$. Thus, if the size of $\mf{G}$ is small then it would be easier to find the complete spectra of $Q_{\com{G}}$. In a group $G$ with bigger $\mf{G}$, even after using \Cref{complete-q}, computing the spectra of $[Q_{\com{G}}/\mf{G}]$ is difficult. Thus, it is worthwhile to find some results on the spectra of $[Q_{\com{G}}/\mf{G}]$. 
			\begin{thm}\label{non_nbd_q}
				Let $G$ be a group. If $u,v\in G$ such that
				\begin{enumerate}
					\item $uv\ne vu$,
					\item   $uw=wu$ if and only if $vw=wu$ for all $w\in G\setminus(\mf{F}_u\cup \mf{F}_v)$.
				\end{enumerate}
				If $|\mf{F}_u|=|\mf{F}_v|=c$ then $|C(u)\cap C(v)|+2(c-1)$ is an eigenvalue of $Q_{\com{G}}$.
			\end{thm}
			\begin{proof}
				To prove the result, it is enough to prove $|C(u)\cap C(v)|+2(c-1)$ is an eigenvalue of $[Q_{\com{G}}/\mf{G}]$. 
				
				Let $\mf{G}=\{\mf{F}_0=Z(G),\ldots, \mf{F}_m\}$.	To show  $$[Q_{\com{G}}/\mf{G}]x_{uv}=\left(|C(u)\cap C(v)|+2(c-1)\right)x_{uv},$$
				where $x_{uv}:\mf{G}\to\mathbb{R}$, is defined by 
				$$x_{uv}(\mf{F}_i)=
				\begin{cases}
					\phantom{-}|\mf{F}_v| &\text{~if~} \mf{F}_i=\mf{F}_u,\\
					- |\mf{F}_u| &\text{~if~} \mf{F}_i=\mf{F}_v,\\
					\phantom{-}0&\text{~otherwise,}
				\end{cases}$$
				we consider the following cases.
				\begin{enumerate}
					\item If $\mf{F}_i\notin \{\mf{F}_u,\mf{F}_v\}$, and $\mf{F}_i$ is adjacent to both $\mf{F}_u,\mf{F}_v$ then 
					\begin{align*}
						([Q_{\com{G}}/\mf{G}]x_{uv})(\mf{F}_i) =|\mf{F}_u|x_{uv}(\mf{F}_u)+|\mf{F}_v|x_{uv}(\mf{F}_v)=0.
					\end{align*}
					\item If $\mf{F}_i$ is not adjacent with $\mf{F}_u$, and $\mf{F}_v$ then either $l_{ij}=0$ or $x_{uv}(\mf{F}_j)=0$. Thus, 
					$	(  [L_\G/\mf{G}]x_{uv})(\mf{F}_i)=0.$
					\item If $\mf{F}_i=\mf{F}_u$, then $q_{ii}=	2(|\mf{F}_i|-1)+\sum\limits_{j(\ne i)=0}^mq_{ij}$. Thus,
					\begin{align*}
						(  [Q_\G/\mf{G}]x_{uv})(\mf{F}_i)&=q_{ii}x_{uv}(\mf{F}_u)=\left(|C(u)\cap C(v)|+2(c-1)\right)x_{uv}(\mf{F}_u).
					\end{align*}
					\item  Similarly, for $\mf{F}_i=\mf{F}_v $, we can show $$	(  [Q_\G/\mf{G}]x_{uv})(\mf{F}_i)=\left(|C(u)\cap C(v)|+2(c-1)\right)x_{uv}(\mf{F}_i).$$
					
				\end{enumerate}
				Thus, the result follows.
			\end{proof}
			If $u,v\in G\setminus Z(G)$ are such that $C(u)\cap C(v)=Z(G)$, then we have the following result.
			\begin{cor}
				If $G$ contains at least two non-equivalent pendent elements $u,v\in G\setminus Z(G)$, with $|\mf{F}_u|=|\mf{F}_v|=c$ then $ |Z(G)|+2(c-1)\in \sigma(Q_{\com{G}})$ .% If the number of the  pendant vertices in $\G_{G/\mf{G}}$ in $G$ is $p$, then the multiplicity of $|Z(G)|$ is at least $p-1$.
				
			\end{cor}
			\begin{cor}
				Let $G$ be a group, and $\mf{G}=\{\mf{F}_0=Z(G),\mf{F}_1,\ldots,\mf{F}_m\}$. If $ \mf{F}_i$ is a pendant in $\G_{G/\mf{G}}$, and $|\mf{F}_i|=c$ for all $i=1,\ldots,m$, then $|Z(G)|+2(c-1)$ is an eigenvalue of $Q_{\com{G}}$ with multiplicity $m-1$.
			\end{cor}
			\begin{ex} Using \Cref{non_nbd_q}, we have the following eigenvalues of $\mf{S}_4$.
				\begin{enumerate}[leftmargin=*]
					\item For $(u,v)=(7,18),(3,11),(6,10)$, we have $|C(u)\cap C(v)|+2(c-1)=4$ is an eigenvalue with multiplicity $3$.
					\item For $(u,v)=(16,12),(16,4),(16,9)$, we have $|Z(G)|+2(c-1)=3$ is an eigenvalue of multiplicity at least $3$.
				\end{enumerate}
			\end{ex}
			\subsection{Adjacency spectra of the commuting graph of a group}
			Consider the matrix $[A_{\com{G}}/\mf{G}]=(a_{ij})_{\mf{F}_i,\mf{F}_j\in\mf{G }}$ associated with $\G_{G/\mf{G}}$, defined by
			$$a_{ij}=\begin{cases}
				|\mf{F}_j|&\text{~if~}i\ne j \text{~and~} \mf{F}_i \text{ is adjacent to }\mf{F}_j,\\
				0&\text{~if~}i\ne j \text{~and~} \mf{F}_i \text{ is not adjacent to }\mf{F}_j,\\
				|\mf{F}_i|-1 &\text{~if~}i=j.
			\end{cases} $$
			Now we show that $\sigma([A_{\com{G}}/\mf{G}])\subseteq \sigma(A_{\com{G}})$.
			\begin{thm}\label{a-contraction}
				Let $G$ be a group. If $\lambda$ is an eigenvalue of the matrix $[A_{\com{G}}/\mf{G}]$ with eigenvector $x:\mf{G}\to\mathbb{R}$, then 
				% $=(a_{ij})_{\mf{F}_i,\mf{F}_j\in\mf{G }}$ of order $|\mf{G}|$, defined by
				% $$a_{ij}=\begin{cases}
					%   |\mf{F}_j|&\text{~if~}i\ne j \text{~and~} \mf{F}_i \text{ is adjacent to }\mf{F}_j,\\
					%    0&\text{~if~}i\ne j \text{~and~} \mf{F}_i \text{ is not adjacent to }\mf{F}_j,\\
					%    |\mf{F}_i|-1 &\text{~if~}i=j,
					% \end{cases} $$
				$\lambda$ is also an eigenvalue of $A_{\com{G}}$ with eigenvector $\ov{x}:{G}\to\mathbb{R}$, defined by $\ov{x}(v)=x(\mf{F}_v)$, for all $v\in V(\G)$.
			\end{thm}
			\begin{proof}
				Let $v\in \mf{F}_i$, for some $i=1,\ldots,m$,
				\begin{align*}
					(A_{\com{G}}\ov{x})(v)&=\sum\limits_{u(\ne v)\in C(v)}\ov{x}(u)\\
					&=\sum\limits_{u(\ne v)\in \mf{F}_i}\ov{x}(u)+\sum\limits_{j(\ne i):j\sim i}\sum\limits_{u\in \mf{F}_j}\ov{x}(u)\\
					&=(|\mf{F}_i|-1)x(\mf{F}_i)+\sum\limits_{j(\ne i):j\sim i}|\mf{F}_j|x(\mf{F}_j)\\
					&=([A_{\com{G}}/\mf{G}]x)(\mf{F}_i).
				\end{align*}
				Thus, the result follows.
			\end{proof}
			The following result provides the complete spectra of $A_{\com{G}}$.
			\begin{thm}[Complete spectrum of $A_{\com{G}}$] \label{complete-a}
				Let $G$ be a group. The complete spectrum of $A_{\com{G}}$ is listed below.
				\begin{enumerate}
					\item For all $v\in G$, $-1$ is an eigenvalue of multiplicity $|\mf{F}_v|-1$.
					\item The remaining eigenvalues are the eigenvalues of $[A_{\com{G}}/\mf{G}] $, that is
					$\sigma([A_{\com{G}}/\mf{G}])\subseteq \sigma(A_{\com{G}})$. If $x\in \Omega_{\lambda}([A_{\com{G}}/\mf{G}])$ for any $\lambda\in \sigma([A_{\com{G}}/\mf{G}])$, then the blow up $\ov{x}\in \Omega_{\lambda}(A_{\com{G}})$.
					% For any eigenvalue $\lambda$, with eigenvector $x:\mf{G}\to\mathbb{R}$ of the matrix $[A_{\com{G}}/\mf{G}]=(a_{ij})_{\mf{F}_i,\mf{F}_j\in\mf{G }}$ of order $|\mf{G}|$, defined by
					%     $$a_{ij}=\begin{cases}
						%       |\mf{F}_j|&\text{~if~}i\ne j \text{~and~} \mf{F}_i \text{ is adjacent to }\mf{F}_j,\\
						%        0&\text{~if~}i\ne j \text{~and~} \mf{F}_i \text{ is not adjacent to }\mf{F}_j,\\
						%        |\mf{F}_i|-1 &\text{~if~}i=j,
						%     \end{cases} $$
					%     $\lambda$ is also an eigenvalue of $A_{\com{G}}$, with eigenvector $\ov{x}:{G}\to\mathbb{R}$, defined by $\ov{x}(v)=x(\mf{F}_v)$, for all $v\in V(\G)$.
				\end{enumerate}
			\end{thm}
			\begin{proof}
				The result follows from \Cref{nbd-eig} and \Cref{a-contraction}. We just need to prove that the list given in this result is complete and that part is similar to the proof of \Cref{complete-q}.
			\end{proof}
			\begin{ex} Now, we use \Cref{complete-a} to provide the following eigenvalues of $\mf{S}_4$.
				\begin{enumerate}[leftmargin=*]
					\item For $v=7,3,10,11,6,18\in \mf{S}_4$, we have $|\mf{F}_v|=2$ (see \Cref{fig:S-4}).
					\item  Similarly, $ |\mf{F}_v|-1=1$ for $v=4,9,12, 16\in \mf{S}_4$.
				\end{enumerate}
				Thus, $-1$ is an eigenvalue of multiplicity at least $10$. Other eigenvalues are the eigenvalues of $[A_{\com{\mf{S}_4}}/\mf{G}]$.
			\end{ex}
			The proof of the following result is similar to that of \Cref{rhoq}. Thus, we state the following result without proof.
			\begin{cor}\label{rhoa}
				For any group $G$, $\rho(A_{\com{G}})=\rho([A_{\com{G}}/\mf{G}]) $. 
			\end{cor}
			Now we prove a result similar to \Cref{non_nbd} and \Cref{non_nbd_q}.
			\begin{thm}\label{non_nbd_a}
				Let $G$ be a group. If $u,v\in G$ such that
				\begin{enumerate}
					\item $uv\ne vu$,
					\item   $uw=wu$ if and only if $vw=wu$ for all $w\in G\setminus(\mf{F}_u\cup \mf{F}_v)$.
				\end{enumerate}
				If $|\mf{F}_u|=|\mf{F}_v|=c$ then $(c-1)$ is an eigenvalue of $A_{\com{G}}$.
			\end{thm}
			\begin{proof}
				Since $A_{\com{G}}=\frac{1}{2}(Q_{\com{G}}-L_{\com{G}})$, by \Cref{non_nbd}, and \Cref{non_nbd_q}, the result follows.
			\end{proof}
			\begin{ex} Using \Cref{non_nbd_a}, we have the following eigenvalues of $\mf{S}_4$.
				\begin{enumerate}[leftmargin=*]
					\item For $(u,v)=(7,18),(3,11),(6,10)$, we have $(c-1)=1$ is an eigenvalue with multiplicity $3$.
					\item For $(u,v)=(16,12),(16,4),(16,9)$, we have $(c-1)=1$ is an eigenvalue of multiplicity at least $3$.
				\end{enumerate}
			\end{ex}
			\section{Group information encoded in spectra of its commuting graph}\label{Groupinfo}
			One key objective of the study of graphs on a group is exploring group properties via these graphs. The commuting graph of a group is not an exception. In this section, we study the group information encoded in the spectra of the commuting graph.
			\subsection{Finding the center of a group} Recall that for $\lambda\in\sigma(L_{\com{G}})$, we denote the eigenspace of $\lambda$ as $ \Omega_{\lambda}(L_{\com{G}})$. For a family of set $\mathcal{S}$, we refer to $S_{\max}\in \mathcal{S}$ as the \emph{maximum element} of $\mathcal{S}$, if $S\subseteq S_{\max}$ for all $S\in \mathcal{S}$.
			\begin{thm}[Spectral method for computation of center of a group]\label{center_LG}
				Let $G$ be a group, and $v_0(\in G)$ be the identity element of $G$. Either $Z(G)=\{v_0\}$ or 
				$Z(G)$ is the maximum element of the  family of sets
				$$\mathcal{S}=\{W\subseteq G:v_0\in W\text{~and~}T_W\text{~is a subspace of~} \Omega_{|G|}(L_{\com{G}})\}.$$
			\end{thm}
			\begin{proof}
				Suppose that $Z(G)\ne\{v_0\}$. Thus, there exists $v(\ne v_0)\in Z(G)$. By \Cref{nbd-eig}, $\{v_0,v\}\in \mathcal{S}$ for all $ v(\ne v_0)\in Z(G)$, and thus, $\mathcal{S}\ne \emptyset$.
				Our claim is if $W\in \mathcal{S} $ then $W\subseteq Z(G)$. Let $W=\{w_0,w_1,\ldots,w_k\}$. We define $y_k:G\to\mathbb{R}$ as $y_i=\chi_{\{v_i\}}-\chi_{\{v_0\}}$ for $i=1,\ldots,k$. Thus, $y_i\in T_W$ for all $i=1,\ldots,k$. Since $W\in \mathcal{S}$, we have $ y_i\in \Omega_{|G|}(L_{\com{G}})$ for all $i=1,\ldots,k$. Thus, $(L_{\com{G}}y_i)=|G|y_i$, and this leads us to 
				\begin{equation}\label{wieqn}
					\sum\limits_{v\in C(w_i)}(y_i(w_i)-y_i(v))=|G|.
				\end{equation}
				Since $(y_i(w_i)-y_i(v))=\begin{cases}
					1&\text{~if~}v\notin\{w_i,w_0\},\\
					0&\text{~if~}v=w_i,\\
					2&\text{~if~}v=w_0,
				\end{cases} $ the following observations show that $ C(w_i)=|G|$, and thus, establish our claim.
				\begin{enumerate}
					\item Since each element of a group belongs to its centralizer, $w_i\in C(w_i)$.
					\item We claim $w_0\in W$. If not, then $\sum\limits_{v\in C(w_i)}(y_i(w_i)-y_i(v))\le |G|-2$, a contradiction to \Cref{wieqn}. Thus, $ w_0\in C(w_k)$. Therefore, $\{w_0,w_k\}\subseteq C(w_k)$.
					\item Now we claim $v\in C(w_k)$ for all $v\in G\setminus \{w_0,w_k\}$. If not, let us assume that there exists $v_i\in G\setminus \{w_0,w_k\}$ such that $v_i\notin C(w_k)$. In that case,
					$\sum\limits_{v\in C(w_i)}(y_i(w_i)-y_i(v))\le |G|-1 $, a contradiction to \Cref{wieqn}. Thus, our assumption is wrong and $v\in C(w_k)$ for all $v\in G\setminus \{w_0,w_k\}$. 
				\end{enumerate}
				Therefore, by (1), (2), and (3), we have $C(w_i)=G$. Similarly, we can show $C(w)=G$ for all $w\in W$, and $W\subseteq Z(G)$. Therefore, $W\subseteq Z(G)$ for all $W\in \mathcal{S} $. Since by \Cref{nbd-eig}, $Z(G)\in\mathcal{S}$. Therefore, $Z(G)$ is the maximum element of $\mathcal{S}$.
			\end{proof}
			For any group $G$, \Cref{center_LG} leads us to the following Algorithm, which can find $Z(G)$ using $L_{\com{G}}$.

			%form $ \sigma(L_{\com{G}})$, and the eigenvectors of $\lambda \in \sigma(L_{\com{G}})$. Let $|G|=n$, and $\omega(L_{\com{G}})=\{(\lambda,z):\lambda\in \sigma(L_{\com{G}}), \text{~and~}L_{\com{G}}z=\lambda z\}$.
			%% This declares a command \Comment
			%% The argument will be surrounded by /* ... */
			\SetKwComment{Comment}{/* }{ */}
			\begin{algorithm}
				\caption{For a group $G$ with identity element $v_0$, $|G|=n$, % and $\omega(L_{\com{G}})=\{(\lambda,z):\lambda\in \sigma(L_{\com{G}}), \text{~and~}L_{\com{G}}z=\lambda z\}$,
					this algorithm returns $Z(G)$.}\label{alg:Z(G)}
				\KwData{$n > 0$, $v_0:$ the identity element of $G$ %$\omega(L_{\com{G}})=\{(\lambda_1,z_1),(\lambda_2,z_2),\ldots,(\lambda_n,z_n)\} $.
				}
				\KwResult{$S=Z(G)$}
				$S \gets \{v_0\}$\;
				%$X \gets x$\;
				%$N \gets n$\;
				\For{$v(\ne v_0)\in G$}{
					$y\gets \chi_{\{v\}}-\chi_{\{v_0\}}$ \Comment*[r]{This for loop computes $Z(G)$}
					\If{$L_{\com{G}}y=n y$ }{
						$S \gets S\cup\{v\}$\;
						%$N \gets \frac{N}{2}$ \Comment*[r]{This is a comment}
					}
					% {\If{$N$ is odd}{
							%     $y \gets y \times X$\;
							%     $N \gets N - 1$\;
							%   }
					}
				\end{algorithm}
				A group is defined on a set, and we refer to the set as the \emph{underlying set} of the group.
				For two group $G_1,G_2$ with the same underlying set, if $\Omega_{|G_1|}(L_{\com{G_1}})=\Omega_{|G_2|}(L_{\com{G_2}})$ then \Cref{center_LG} leads us to the following result.
				\begin{cor}
					Let $G_1$ and $G_2$ be two groups with the same underlying set. If $\Omega_{|G_1|}(L_{\com{G_1}})=\Omega_{|G_2|}(L_{\com{G_2}})$ then $Z(G_1)=Z(G_2)$.
				\end{cor}
				\begin{proof}
					If $\Omega_{|G_1|}(L_{\com{G_1}})=\Omega_{|G_2|}(L_{\com{G_2}})$  then the set $\mathcal{S}$ would be the same for both the groups and thus, by \Cref{center_LG}, $Z(G_1)=Z(G_2)$.
				\end{proof}
				\begin{ex}\label{s4}
					
					\begin{enumerate}[leftmargin=*]
						Now we use \Cref{center_LG} to find center of some groups.
						\item   Let us consider $\com{\mf{S}_4} $ (\Cref{fig:S-4}). %The spectrum of the Laplacian of $\com{\mf{S}_4} $ is $\sigma(L_{\com{\mf{S}_4}})=\{0, 1, 1, 1, 1, 2, 2, 2, 3, 3, 3, 3, 4, 4, 
						%$\noindent $
						% 4, 4, 4, 4, 6, 24, 5-\sqrt{13}, 5-\sqrt{13}, \sqrt{13}+5, \sqrt{13}+5  \}.$
						Since for all $v(\ne 1)\in \mf{S}_4$, $L_{\com{\mf{S}_4}}(\chi_{\{v\}}-\chi_{\{1\}})\ne |G|(\chi_{\{v\}}-\chi_{\{1\}})$, by Algorithm (\ref{alg:Z(G)}), we have $Z(\mf{S}_4)=\{1\}$.
						%  The $Z(\mf{S}_4)=\{1\}$, and $|Z(\mf{S}_4)|\in \sigma(L_{\mf{S}_4})$, because $\mf{S}_4$ contains pendants. 
						\item For $\G=\com{Q_8}=\com{D_8}$. Since other than the identity $v_0$, there exists only one $v=x^2$ in $D_8$ $(v=a$ in $Q_8)$ such that $L_{\G}(\chi_{\{v\}}-\chi_{\{v_0\}})= |G|(\chi_{\{v\}}-\chi_{\{v_0\}})$. Thus, by Algorithm (\ref{alg:Z(G)}), we have $Z(G)=\{v_0,v\}$. 
					\end{enumerate}
					
				\end{ex}
				
				\subsection{How far a non-commuting group is from being commutative}
				If a group $G$ is commutative, then $\com{G}$ is a complete graph. Thus, the spectra of $L_{\com{G}}$ contains only two distinct real ($0$, and $|G|$). This is the only case where the number of distinct eigenvalues of  $L_{\com{G}}$ is $2$. If $G$ is not commutating, then the number of distinct eigenvalues of  $L_{\com{G}}$ is always greater than $2$. Thus, the number of distinct eigenvalues of  $L_{\com{G}}$ measures how far $G$ is from being commutative. The following result provides a class of groups such that the Laplacian spectrum of the commuting graph has $4$ distinct elements.
				\begin{prop}\label{zg-prop}
					Let $G$ be a group.
					If for all $u,v\in G\setminus Z(G)$,
					\begin{itemize}
						\item[(i)]$|C(u)|=|C(v)|$, and
						\item[(ii)] $\text{~either~} C(u)=C(v) \text{~or~} C(u)\cap C(v)=Z(G),$
					\end{itemize}
					then the spectrum of $L_{\com{G}}$ has $4$ distinct elements.
				\end{prop}
				\begin{proof}
					Sine $|C(u)|=|C(v)|=m$(say), for all $u,v\in G\setminus Z(G)$, by \Cref{all-pendant}, we have the $4$ distinct eigenvalues are $|G|,m,|Z(G)|,0$.
				\end{proof}

				{If along with the condition (i) and (ii) of \Cref{zg-prop}, $|\mf{F}_u|=1$ for all $u\in G\setminus Z(G)$ then by \Cref{all-pendant}, $L_{\com{G}}$ has only three distinct eigenvalues. However, we are unable to find a suitable group in support of it. Thus, we land into the question whether such a group $G$ exists that $L_{\com{G}}$ has only three distinct eigenvalues.}
				
				For a graph $\G$, the least non-zero eigenvalue of $L_{\G}$ is called the \emph{algebraic connectivity} of $\G$ \cite{fiedler1973algebraic}. The algebraic connectivity measures how well connected $\G$ is. For a complete graph on $n$ vertices, the algebraic connectivity is $n$, which is the maximum, and it is $0$ for a disconnected graph. Thus, for a group $G$, the algebraic connectivity of $\com{G}$ is another measure of how far $G$ is from being a commutative group. If $G$ is commutative, then the algebraic connectivity is $|G|$, which is the maximum. For a group $G$ that satisfy \Cref{con}, the algebraic connectivity is $|Z(G)|$ (see \Cref{all-pendant}).
				Thus, we have the following result.
				\begin{prop}\label{iso}
					For a group $G$ that satisfy \Cref{con}, the algebraic connectivity of the graph $\com{G}$ is given by the cardinality
					of the center of the group $G$.
				\end{prop}
				\begin{proof}
					The result follows from the complete list of eigenvalues given in \Cref{all-pendant}.
				\end{proof}
				Evidently, by \Cref{iso}, the isoperimetric number of $\mathcal C_{\mathfrak S_3}$ is $|Z(\mathfrak S_3)|=1.$
				%, which is the lowest among all the groups with at least one $v\in G\setminus Z(G)$, with $|\mf{F}_u|>ge 1$.
				\section{Graph Invariants for Commuting Graphs}\label{app}
				Any graph property is said to be a graph invariant if it remains unchanged under graph isomorphisms.
				The goal of this section is to express some graph invariants of the commuting graph associated with a finite group in terms of its Laplacian spectrum.
				\subsection{Minimum degree, Independent number and Clique Number}
				%We have given the complete list of eigenvalues of a commuting graph. Using the complete list and some existing results we can conclude the following
				An independent set of a graph $\Gamma(V,E)$ is a subset $S \subseteq V$ such that if $v_1,v_2\in S$ then $v_1$ and $v_2$ are not adjacent. The cardinality of the largest independent set in $V$ is called the independent number $\alpha(\Gamma)$ of the graph $\Gamma$.
				
				In the commuting graph $\mathcal{C}_G$, the degree of a vertex $v\in G $ is $d(v)=|C(v)|-1$. So the maximum degree is always given by $|G|-1.$ Let $c = \displaystyle\text{min}_{u \in G}|C(u)|.$ Then the minimum degree in $\mathcal C_G$ is $c-1.$ From \cite[Theorem 2.6, Theorem 3.2]{MR2344135}, one has $$1\le \alpha(\m C_G) \le |G| - c+1.$$ 
				For an abelian group $G,$ we have $c =|G|,$ so it follows that $\alpha(\m C_G) =1.$
				Suppose that the group $G$ satisfies Equation \eqref{con}.  By  \Cref{all-pendant}, $d(v) +1$ is an eigenvalue with multiplicity $d(v) - |Z(G)|.$ If $v$ commutes with at least one element $u (\neq v)$ which is not in $Z(G),$ then $d(v) \neq |Z(G)|.$ This implies that $d(v)+1$ is an eigenvalue of the Laplacian matrix of $\mathcal{C}_G$ with multiplicity $d(v) - |Z(G)| \geq 1$. Consider the subset 
				\begin{equation}
					\label{cu1} C(G)=\{v\in G: d(v) - |Z(G)| >0\}.
				\end{equation}   
				If $C(G)=G,$ then by \Cref{all-pendant}, $c=\l_3$ and  $1\le \alpha(\m C_G)\le |G| - \l_3+1$, where $\l_3$ is the third least eigenvalue of the Laplacian matrix of the graph $\mathcal C_G.$ Therefore, we have the following result.
				
				\begin{prop}\label{new_prop1}
					Suppose that $G$ is a group that satisfies Equation \eqref{con}. If $ C(G)=\{v\in G: d(v) - |Z(G)| >0\}=G$, then the independent number $\alpha(\m C_G)$ lies in the following interval: $$ 1\le \alpha(\m C_G)\le |G| - \l_3+1,$$ where $\l_3$ is the third least eigenvalue of the Laplacian matrix of the graph $\mathcal C_G.$
				\end{prop}
				%\item  From \cite[Theorem 2.6, Theorem 3.2]{MR2344135}, one has 
				%$$ 1\le \alpha(\m C_G)\le |G| - c+1,$$ where $c = \displaystyle\text{min}_{u \in G}|C(u)|.$ In case every element of the group $G$ commutes with at least one element except $Z(G)$ and itself, then $c$ is given by the third least eigenvalue of the associated Laplacian matrix. 
				A clique $S$ in a graph $\Gamma (V,E)$ is a subset $S \subseteq V$ such that any two distinct vertices $v_1, v_2 \in S$ are adjacent to each other. The cardinality of the maximum clique is called the clique number, $\omega (\Gamma)$. 
				
				Clearly, $\omega(\mathcal{C}_G)=|G|$, whenever $G$ is an abelian group. Suppose that $G$ is a non-abelian group satisfying Equation \ref{con}. For every $u\in G\setminus Z(G)$, the centralizer $C(u)$ is a clique in $\mathcal{C}_G$. From the construction of $\mathcal C_G,$ we conclude that $\omega(\mathcal{C}_G)=\max\limits_{u\in G\setminus Z(G)}|C(u)| $. By Theorem \Cref{all-pendant} (3), the following result holds.
				\begin{prop}
					If $G$ is a non-abelian group satisfying Equation \eqref{con}, then the clique number $\omega(\mathcal{C}_G)$ is given by 
					\begin{enumerate}
						\item[(i)] the second largest eigenvalue of the Laplacian matrix of the commuting graph $\mathcal{C}_G$, if $C(G)\supsetneq Z(G)$ and
						\item[(ii)] $|Z(G)|+1,$ if $Z(G) = C(G),$
					\end{enumerate}
					where $C(G)$ is as defined in Equation \eqref{cu1}.
				\end{prop}

				\subsection{Mean distance and Graph diameter}
				The distance between two vertices of a graph is the number of edges in a shortest path connecting them. The $ij$-th element of the graph distance matrix is given by the distance of the vertices $v_i$ and $v_j.$ The diameter of the graph is the maximum element of the graph distance matrix. Moreover, the mean distance is the average of all elements of the graph distance matrix \cite{MR485476}.
				
				Let enumerate the elements of the group $G$ by $\{v_1,\ldots,v_{|G|}\}.$ Suppose that the length of the shortest path between any two vertices $v_i$ and $v_j$ of $\m C_G$ is denoted by $\gamma_{ij}.$ The graph distance matrix is the square matrix $\big(\!\!\!\big(\gamma_{ij}\big)\!\!\!\big)_{i,j=1}^{|G|}$. For any abelian group $G$,  $$\gamma_{ij} = \begin{cases}
					1, &\text{~for~} i \neq j,\\
					0, &\text{~otherwise.}
				\end{cases}$$ The mean distance and diameter of the graph $\m C_G$ are given by $\frac{|G| -1}{|G|}$ and $1$, respectively.
				Suppose that $G$ is a finite non-abelian group which satisfies Equation \eqref{con}. By the construction of $\m C_G,$ it is clear that the following hold: 
				$$\gamma_{ij} = \begin{cases}
					0 &\text{if } v_i = v_j,\\
					1 &\text{if } v_i \text{ or } v_j \in Z(G),\\
					1 &\text{if } v_i,v_j \notin Z(G), v_i,v_j \in C(u) \text{ for } u \in G,\\
					2 &\text{if } \nexists u \in G \text{ such that } v_i, v_j \in C(u).
				\end{cases}$$
				So the mean distance of $\m C_G =\frac{1}{|G|^2} \sum\limits_{i,j=1}^{|G|} \gamma_{ij}$ and the diameter of $\m C_G = \rm{\max\limits_{i,j}} \gamma_{ij} = 2.$ Now we have \Bea \sum_{i,j=1}^{|G|} \gamma_{ij} = \sum_{v \in G} (|C(v)| + 2|G\setminus C(v)| -1) &=& 2|G|^2 -|G| -\sum_{v \in G} |C(v)|\\
				&=& 2|G|^2 -|G| - |Z(G)||G| - \sum_{v \in G\setminus Z(G)} |C(v)|\\
				&=& 2|G|^2 -2|G| -\sum_{i} n_i\l_i,\Eea
				where each $\l_i$ is the eigenvalue of the Laplacian $L$ with multiplicity $n_i.$ Moreover, $\sum_{i} n_i\l_i = \tr(L) = \sum_{v \in G}d(v),$ where $d(v)$ denotes the degree of $v \in G$ in the graph $\m C_G.$ Therefore, the mean distance of $ C_G$ is $\frac{2|G|^2 -2|G| - \sum_{v \in G}d(v)}{|G|^2}$ and one has the following result.
				
				%$$ \text{the mean distance of~} \m C_G =\frac{2|G|^2 -2|G| - \sum_{v \in G}d(v)}{|G|^2}.$$
				\begin{prop}
					If $G$ is a finite non-abelian group which satisfies Equation \eqref{con} then the mean distance of $ C_G$ is 
					$$\frac{2|G|^2 -2|G| - \sum_{v \in G}d(v)}{|G|^2},$$
					where $d(v)$ denotes the degree of $v \in G.$ \end{prop}
				\section{Acknowledgement} The first named author acknowledges the financial support from  IoE-IISc. The work of the second named author is supported by the University Grants Commission, India. %Authors would like to thank Dr. Subrata Shyam Roy for his numerous helpful comments to improve the exposition of this article.
				\section*{Conflict of interest}
				We want to declare that there are no known conflicts of interest with this work. 
				%We wish to confirm that there are no known conflicts of interest associated with this work. 
		%		\bibliographystyle{siam}
		%		\bibliography{ref}

			\end{document}